\documentclass[11pt,reqno]{amsart}
\usepackage[margin=1in]{geometry}
\usepackage{amssymb,mathtools,mathrsfs,enumitem}
\usepackage[hidelinks]{hyperref}
\newtheorem{theorem}{Theorem}[section]
\newtheorem{lemma}[theorem]{Lemma}
\newtheorem{proposition}[theorem]{Proposition}
\newtheorem{corollary}[theorem]{Corollary}
\newcommand{\F}{\mathbf F}
\DeclareMathOperator{\GL}{GL}
\DeclareMathOperator{\Aut}{Aut}
\DeclareMathOperator{\Inn}{Inn}
\DeclareMathOperator{\Out}{Out}
\DeclareMathOperator{\Hom}{Hom}
\DeclareMathOperator{\im}{im}
\DeclareMathOperator{\rank}{rank}
\DeclareMathOperator{\cl}{cl}
\title{Almost all finite groups are $2$-groups of class two}
\author{Ning Su}
\author{Binzhou Xia}
\date{}

\begin{document}

\begin{abstract}
We prove that almost all finite groups are $2$-groups of nilpotency class two, when groups of order at most $x$ are counted up to isomorphism and $x\to\infty$. More precisely, uniformly for $2^m\leq x<2^{m+1}$, all but an $O(2^{-cm^2})$ proportion have order $2^m$ and class two, for some absolute constant $c>0$. For each fixed prime $p$, we give a formula for the number of groups of order $p^m$, with relative error $O(p^{-m/3})$ as $m\to\infty$. This refines the classical logarithmic estimates for the number of $p$-groups. For every fixed prime $p$ and sufficiently large $m$, we also prove that all but a $p^{-m^2/300}$ proportion of groups of order $p^m$ have central elementary abelian Frattini subgroup. For fixed $p$, the distribution of the minimum number of generators among groups of order $p^m$ is asymptotically supported on one or two adjacent values, with explicit probabilities depending on $m$ modulo three.
\end{abstract}

\maketitle

\section{Introduction}\label{sec:introduction}

The enumeration of finite groups raises the fundamental question of what a typical finite group looks like.
Throughout, groups are counted up to isomorphism.
Let $f(n)$ be the number of isomorphism classes of groups of order $n$, and let
\[
F(x)=\sum_{1\leq n\leq x}f(n).
\]
A central problem in the enumeration of finite groups is to understand the asymptotic properties of the isomorphism classes counted by $F(x)$.

The rapid growth of $p$-groups provides important evidence for the expected behavior of finite-group enumeration.
The estimates of Higman~\cite{Higman1960} and Sims~\cite{Sims1965}, with the refinement of Newman and Seeley (see~\cite[Theorem~5.8]{BNV2007}), give, for each fixed prime $p$,
\begin{equation}\label{eq:HigmanSims}
\log_p f(p^m)=\frac{2m^3}{27}+O(m^{\frac{5}{2}}).
\end{equation}
Together with general upper bounds for the number of groups of a given order in \cite{Pyber1993} and \cite[Chapter~16]{BNV2007}, the estimates for $f(p^m)$ support the conjecture that almost all finite groups have prime-power order, and more specifically are $2$-groups.
A formally weaker question, posed in \cite[Question~1]{Mann1999} and \cite[Question~22.19]{BNV2007}, is whether almost all finite groups are nilpotent, namely whether
\[
\frac{\#\{G:\ |G|\leq x,\ G\text{ nilpotent}\}}{F(x)}
\longrightarrow1
\]
as $x\to\infty$.

We prove the stronger statement that almost all finite groups are $2$-groups of nilpotency class two.
More precisely, almost all finite groups have the largest power-of-two order allowed by the bound $x$.

\begin{theorem}\label{thm:main}
There is an absolute constant $c>0$ such that, uniformly for $2^m\leq x<2^{m+1}$ as $m\to\infty$,
\begin{equation}\label{eq:dyadic-asymptotic}
F(x)=f(2^m)\bigl(1+O(2^{-cm^2})\bigr).
\end{equation}
Moreover, all but an $O(2^{-cm^2})$ proportion of the isomorphism classes counted by $F(x)$ have order $2^m$ and nilpotency class two.
\end{theorem}

Thus the asymptotic enumeration of finite groups of order up to $x$ is determined by the groups of the largest power-of-two order below $x$.
Equivalently, after decreasing $c$ if necessary,
\begin{equation}\label{eq:cumulative-asymptotic}
F(x)=f(2^{\lfloor\log_2x\rfloor})\bigl(1+O(2^{-c(\log_2x)^2})\bigr).
\end{equation}
In particular, for sufficiently large $m$, there are more groups of order $2^m$ than of all smaller orders combined.
As a consequence of Theorem~\ref{thm:main}, Corollary~\ref{cor:cumulative-domination} gives the quantitative bound
\[
\frac{F(2^m-1)}{f(2^m)}=O(2^{-cm^2})
\]
as $m\to\infty$. For sufficiently large $m$, this proves $f(n)<f(2^m)$ for all $1\leq n<2^m$, which was a conjecture attributed to Erd\H{o}s and to Higman (see~\cite[Question~22.16]{BNV2007}).
It also provides an affirmative answer to the first part of~\cite[Question~22.18]{BNV2007} asking whether $F(2^m-1)<f(2^m)$ for sufficiently large $m$.
Moreover, the second part of~\cite[Question~22.18]{BNV2007} asks whether $F(2^m-1)\sim f(2^{m-1})$ as $m\to\infty$, and applying Theorem~\ref{thm:main} at $x=2^m-1$ gives the stronger estimate
\[
F(2^m-1)=f(2^{m-1})\bigl(1+O(2^{-c(m-1)^2})\bigr).
\]

Theorem~\ref{thm:main} reduces the global enumeration problem to determining the number of $2$-groups, or more generally, groups of prime-power order. The logarithmic estimate~\eqref{eq:HigmanSims} determines only the leading term of $\log_p f(p^m)$, whereas our Theorem~\ref{thm:sharp} below determines $f(p^m)$ itself up to a relative error tending to zero. For a fixed prime $p$, let
\[
P_p(m)=f(p^m)\quad\text{and}\quad g_p=\prod_{i\geq1}(1-p^{-i}).
\]
Unless stated otherwise, implied constants and sufficiently large thresholds may depend on parameters explicitly held fixed.

\begin{theorem}\label{thm:sharp}
Let $p$ be a fixed prime.
As $m\to\infty$,
\[
P_p(m)=M_p(m)\bigl(1+O(p^{-\frac{m}{3}})\bigr),
\]
where
\begin{equation}\label{eq:sharp-main-term}
M_p(m)=
\begin{cases}
p^{\frac{2}{27}m^3-\frac{4}{9}m^2+\frac{m}{3}-2}/g_p^2&\text{if }m\equiv0\pmod3\\
p^{\frac{2}{27}m^3-\frac{4}{9}m^2+\frac{m}{3}-\frac{26}{27}}/g_p^2&\text{if }m\equiv1\pmod3\\
(1+p^{-3})p^{\frac{2}{27}m^3-\frac{4}{9}m^2+\frac{2m}{9}-\frac{7}{27}}/g_p^2&\text{if }m\equiv2\pmod3.
\end{cases}
\end{equation}
\end{theorem}

Taking $p=2$ and $m=\lfloor\log_2x\rfloor$, Theorems~\ref{thm:main} and~\ref{thm:sharp} approximate $F(x)$ by $M_2(m)$ with relative error $O(2^{-m/3})$ as $x\to\infty$.

The sharp asymptotic in Theorem~\ref{thm:sharp} is accompanied by a structural concentration phenomenon:
for each fixed prime $p$, almost all isomorphism classes of groups of order $p^m$ have central elementary abelian Frattini subgroup.
To be more precise, let $\Phi(G)$ denote the Frattini subgroup of a finite $p$-group $G$, and define
\[
\begin{aligned}
R_p(m)&=\#\{G:|G|=p^m,\ \Phi(G)\leq Z(G),\ \Phi(G)^p=1\},\\
Q_p(m)&=\#\{G:|G|=p^m,\ \cl(G)\leq2\}.
\end{aligned}
\]
Here $\cl(G)$ denotes the nilpotency class.
Mann~\cite[Question~3]{Mann1999} asked whether almost all $p$-groups have nilpotency class two, or at least whether such groups have positive density among the groups of order $p^m$. We prove a stronger structural concentration result as follows.

\begin{theorem}\label{thm:structure}
Let $p$ be a fixed prime.
For all sufficiently large integers $m$,
\[
R_p(m)\leq Q_p(m)\leq P_p(m)
\leq(1+p^{-\frac{m^2}{300}})R_p(m).
\]
In particular,
\[
\lim_{m\to\infty}\frac{R_p(m)}{P_p(m)}
=
\lim_{m\to\infty}\frac{Q_p(m)}{P_p(m)}
=1.
\]
\end{theorem}

Since the number of abelian groups of order $p^m$ is the partition number of $m$, which is at most $2^{m-1}$, abelian groups form a negligible proportion of the groups of order $p^m$.
Hence Theorem~\ref{thm:structure} implies that almost all groups of order $p^m$ have nilpotency class exactly two.
Theorem~\ref{thm:structure} also shows that the asymptotic in Theorem~\ref{thm:sharp} is already accounted for by groups of order $p^m$ with central elementary abelian Frattini subgroup.

For the family counted by $R_p(m)$, Bläser, Li, Qiao and Rogovskyy \cite[Theorem~1.3]{BLQR} proved
\[
\log_p R_p(m)=\frac{2m^3}{27}-\frac{4m^2}{9}+O(m).
\]
Theorems~\ref{thm:sharp} and~\ref{thm:structure} refine this logarithmic estimate for $R_p(m)$ by giving a formula for the number $P_p(m)$ of all groups of order $p^m$, with relative error $O(p^{-m/3})$.

Theorem~\ref{thm:structure} also allows us to study finer statistics of a random group of order $p^m$.
A fundamental invariant of a finite $p$-group is its minimum number of generators, determined by the Frattini quotient through the Burnside basis theorem. For a finite $p$-group $G$, let
\[
d(G)=\dim_{\F_p}G/\Phi(G)
\]
denote the rank of $G$. The distribution of $d(G)$ describes the typical rank of the $p$-groups contributing to the asymptotic count.
Let $\delta_i$ denote the point mass at $i$, and define, for $m\geq3$,
\begin{equation}\label{eq:generator-law}
\pi_{p,m}=
\begin{cases}
\delta_{2a-1}&\text{if }m=3a\\[1mm]
\delta_{2a}&\text{if }m=3a+1\\[1mm]
(p^3\delta_{2a+1}+\delta_{2a})/(p^3+1)&\text{if }m=3a+2.
\end{cases}
\end{equation}
For $X\in\{P_p,Q_p,R_p\}$, let $\mu_{X,p,m}$ denote the distribution of $d(G)$ on the uniformly chosen isomorphism classes counted by $X(m)$.
We use the total variation distance
\[
d_{\mathrm{TV}}(\mu,\nu)=\frac12\sum_{i\geq0}|\mu(i)-\nu(i)|.
\]

\begin{theorem}\label{thm:generator}
Let $p$ be a fixed prime and let $X\in\{P_p,Q_p,R_p\}$.
As $m\to\infty$,
\[
d_{\mathrm{TV}}(\mu_{X,p,m},\pi_{p,m})=O(p^{-\frac{m}{3}}).
\]
\end{theorem}

Thus the minimum number of generators is concentrated at one value when $m\equiv0,1\pmod3$.
When $m\equiv2\pmod3$, two adjacent values occur with limiting probabilities
\[
\frac{p^3}{p^3+1}\quad\text{and}\quad\frac{1}{p^3+1}.
\]

Another application of the structural theorem concerns group isomorphism testing.
Grochow, Ivanyos, Qiao and Sun \cite[Theorem~1.1]{GIQS2026} give an algorithm that tests isomorphism of two $2$-groups of order $N$ with central elementary abelian Frattini subgroup, given by their Cayley tables, in time
\[
N^{O((\log N)^{\frac{1}{2}})}.
\]
By Theorems~\ref{thm:main} and~\ref{thm:structure}, the isomorphism-testing algorithm applies to a family of density one among all isomorphism classes of finite groups, ordered by group order.

The rest of the paper is organized as follows. Sections~\ref{sec:roots}--\ref{sec:density} establish the structural concentration theorem.
Section~\ref{sec:sharp-enumeration} combines the structural concentration theorem with a sharp orbit count to prove the counting formula in Theorem~\ref{thm:sharp} and the generator distribution in Theorem~\ref{thm:generator}.
Section~\ref{sec:cumulative-enumeration} combines the $2$-group estimates with bounds for groups whose order has an odd divisor to prove the global enumeration theorem.

\section*{AI Declaration}

The work reported in this paper was carried out in collaboration with the AI system GPT-6, which we used extensively in developing the arguments and preparing the manuscript.
We initially studied the proportion of groups of nilpotency class two among finite $p$-groups of a given order.
During discussions with GPT-6, we focused on a family of such groups that appeared particularly amenable to counting.
Further work with GPT-6 led to a candidate proof that, for each fixed prime $p$, the proportion of isomorphism classes of groups of order $p^m$ having nilpotency class two tends to one as $m$ tends to infinity.

While examining this argument, we noticed that estimates involving the automorphism groups of the $2$-groups in this family might also provide an approach to the enumeration of general finite groups.
Further discussions with GPT-6 led to a candidate proof that almost all finite groups are $2$-groups.
Most of the technical details of the initial draft were written by GPT-6.
We are continuing to revise the manuscript to improve its readability.

\section{Roots and their parameters}\label{sec:roots}

Throughout, $p$ is a prime.
We use the commutator convention and lower central series defined by
\[
[x,y]=x^{-1}y^{-1}xy,\quad
\gamma_1(G)=G,\quad \gamma_{i+1}(G)=[\gamma_i(G),G],
\]
and the lower exponent-$p$ central series
\[
\gamma_1^{(p)}(G)=G,\quad \gamma_{i+1}^{(p)}(G)=\gamma_i^{(p)}(G)^p[\gamma_i^{(p)}(G),G].
\]
Here $H^p$ is the subgroup generated by the $p$th powers in $H$.
Write $G'=[G,G]$, and let $\Phi(G)$ be the intersection of the maximal proper subgroups, with $\Phi(1)=1$.
The Frattini formula is
\[
\Phi(G)=G^p[G,G]=\gamma_2^{(p)}(G),\quad
d(G)=\dim_{\F_p}G/\Phi(G).
\]
Unless otherwise stated, all vector spaces and tensor products are over $\F_p$.
Let $D_d=d(d+1)/2$.
We call a finite $p$-group $G$ a \emph{root} if $\gamma_3^{(p)}(G)=1$.
The following observation identifies the roots with the groups counted by $R_p(m)$.

\begin{lemma}\label{lem:root-characterization}
For every finite $p$-group $G$,
\[
\gamma_3^{(p)}(G)=1\quad\Longleftrightarrow\quad
\Phi(G)\leq Z(G)\ \text{and}\ \Phi(G)^p=1.
\]
Thus $\Phi(G)$ is central elementary abelian, with the trivial subgroup allowed, and $1\leq R_p(m)\leq Q_p(m)\leq P_p(m)$ for $m\geq1$.
\end{lemma}
\begin{proof}
The equivalence follows from $\gamma_3^{(p)}(G)=\Phi(G)^p[\Phi(G),G]$.
Since $G'\leq\Phi(G)$, every root has class at most two, and the elementary abelian group of order $p^m$ gives $R_p(m)\geq1$.
\end{proof}

For a nontrivial finite $p$-group define its canonical root by $Q(G)=G/\gamma_3^{(p)}(G)$.
Powers and commutators commute with epimorphic images, so $\gamma_3^{(p)}(Q(G))=1$.
If $|G|=p^m$, set
\[
d=d(G),\quad s=\dim\Phi(Q(G)),\quad t=\log_p|\gamma_3^{(p)}(G)|.
\]
Since $\gamma_3^{(p)}(G)\leq\Phi(G)$,
\[
\Phi(Q(G))=\Phi(G)/\gamma_3^{(p)}(G),\quad
Q(G)/\Phi(Q(G))\cong G/\Phi(G),\quad m=d+s+t.
\]
Let $\mathcal R_p(d,s)$ be a set of representatives of roots with $d(Q)=d$ and $\dim\Phi(Q)=s$, and write $B_p(d,s)=|\mathcal R_p(d,s)|$.
For $t\geq0$ let
\[
D_p(Q,t)=\#\{G:G/\gamma_3^{(p)}(G)\cong Q,\ |\gamma_3^{(p)}(G)|=p^t\}.
\]
Finally, set
\[
N_p(m)=\#\{G:|G|=p^m,\ \gamma_3^{(p)}(G)\ne1\}.
\]
Thus $N_p(m)=P_p(m)-R_p(m)$.
Since $\gamma_3(G)\leq[\gamma_2^{(p)}(G),G]\leq \gamma_3^{(p)}(G)$, we also have $P_p(m)-Q_p(m)\leq N_p(m)$.
A nonroot may still have nilpotency class two.

\subsection{The free root and its quotients}
\label{subsec:free-root}

\begin{lemma}\label{lem:free-root}
For $r\geq1$, let $F_r$ be the free group on generators $x_1,\ldots,x_r$ and let
\[
U_r=F_r/(F_r^{p^2}\gamma_2(F_r)^p\gamma_3(F_r)),\quad
W_r=\Phi(U_r).
\]
Then $W_r$ is central elementary abelian with basis $x_i^p,[x_j,x_i]$ ($i<j$), and $U_r/W_r\cong\F_p^r$.
Every root with parameters $(r,s)$ is $U_r/N$ for a codimension-$s$ subspace $N\leq W_r$.
The map $\Aut(U_r)\to\GL(r,p)$ is onto, and its kernel fixes $W_r$ pointwise; write $T_P$ for the resulting action of $P\in\GL(r,p)$ on $W_r$.
\end{lemma}
\begin{proof}
Collection using $ab=ba[a,b]$ gives the normal forms
\[
x_1^{a_1}\cdots x_r^{a_r}\prod_{i<j}[x_j,x_i]^{b_{ij}},
\quad 0\leq a_i<p^2,\quad0\leq b_{ij}<p.
\]
Thus $U_r$ is finite of order at most $p^{r+D_r}$.
The subgroup $W_r$ generated by $x_i^p$ and $[x_j,x_i]$ is central elementary abelian: $[x_i^p,y]=[x_i,y]^p=1$.

These $D_r$ generators are independent.
In any relation among them, map one $x_i$ to a generator of the cyclic group of order $p^2$ and all other generators to one.
The resulting cyclic-group relation forces the coefficient of $x_i^p$ to vanish.
Next, for each $i<j$, map the selected two generators to the adjacent elementary unipotent matrices in $\operatorname{UT}_3(\F_p)$ and all others to one.
The group $\operatorname{UT}_3(\F_p)$ satisfies the defining identities of $U_r$: if $A^3=0$, then
\[
(I+A)^{p^2}=I+p^2A+\binom{p^2}{2}A^2=I,
\]
including $p=2$.
The selected commutator has order $p$, forcing its coefficient to vanish.
Hence $\dim W_r=D_r$.
Mapping $x_iW_r$ to the standard basis of $\F_p^r$ proves $U_r/W_r\cong\F_p^r$ and $|U_r|=p^{r+D_r}$.

The Frattini formula gives $\Phi(U_r)=W_r$.
We shall also use Burnside's basis theorem: a subset of a finite $p$-group generates the group if and only if its image generates the Frattini quotient.

For a codimension-$s$ subspace $N\leq W_r$, the quotient satisfies
\[
\Phi(U_r/N)=W_r/N,\quad |U_r/N|=p^{r+s},\quad d(U_r/N)=r.
\]
Conversely, a root $Q$ counted by $B_p(r,s)$ is generated by lifts of a basis of its Frattini quotient.
Centrality and exponent $p$ of $\Phi(Q)$ give $\gamma_3(Q)=1$, $[Q,Q]^p=1$ and $x^{p^2}=1$.
Thus those lifts define an epimorphism $U_r\to Q$.
Its kernel lies in $W_r$, since the map on Frattini quotients is an isomorphism, and has codimension $s$.
These quotients exhaust the root classes.

Lift the columns of $P\in\GL(r,p)$ to elements of $U_r$.
Freeness gives an endomorphism, and the Frattini generation property makes it surjective, hence an automorphism.
An element of the kernel has $x_i\mapsto x_ih_i$, $h_i\in W_r$, and fixes $W_r$ pointwise because
\[
(x_ih_i)^p=x_i^p,\quad [x_jh_j,x_ih_i]=[x_j,x_i].
\]
Thus the restriction $T_P$ of a lift of $P$ to $W_r$ is independent of the chosen lift.
\end{proof}

\begin{lemma}\label{lem:root-automorphisms}
For a root $G$, let $V=G/\Phi(G)$, $W=\Phi(G)$, $d=\dim V$, and $s=\dim W$.
The commutator and $p$th-power operations induce maps $\beta_G:V\times V\to W$ and $\pi_G:V\to W$, given by
\[
\beta_G(xW,yW)=[x,y],\quad \pi_G(xW)=x^p.
\]
Write $y_G=(\beta_G,\pi_G)$.
Every pair of vector-space isomorphisms between the corresponding spaces of two roots that preserves the commutator and $p$th-power maps lifts to a group isomorphism.
In particular,
\[
1\longrightarrow T(G)\longrightarrow\Aut(G)
\longrightarrow\operatorname{Stab}_{\GL(V)\times\GL(W)}(y_G)
\longrightarrow1,
\]
where the kernel is naturally the additive group $\Hom(V,W)$ via
\[
f\longmapsto\tau_f,\quad \tau_f(x)=x f(xW).
\]
Thus $T(G)$ is elementary abelian of order $p^{ds}$, and $\Aut(G)=T(G)$ exactly when the joint stabilizer is trivial.
For $s=0$, $G=V$, $T(G)=1$, and the stabilizer is $\GL(V)$.
\end{lemma}
\begin{proof}
Choose Frattini bases and use Lemma~\ref{lem:free-root} to write the two roots as $U_d/\ker\theta_i$, with surjections $\theta_i:W_d\to W_i$.
The power and commutator maps of the two roots determine $\theta_i$ on the displayed basis of $W_d$.
A compatible pair $(P,Q)$ therefore gives $\theta_2T_P=Q\theta_1$, and hence $T_P(\ker\theta_1)=\ker\theta_2$.
A lift of $P$ in $\Aut(U_d)$ descends to the required isomorphism.
The lifting construction also proves surjectivity in the exact sequence.
The case $s=0$ is the usual action on $V$; if $d=0$, the Frattini generation property makes the group trivial.

Since $W$ is central elementary abelian, linearity gives $\tau_f(xy)=\tau_f(x)\tau_f(y)$ and $\tau_f\tau_h=\tau_{f+h}$, so the maps $\tau_f$ are automorphisms fixing $W$ and inducing the identity on $V$, with inverse $\tau_{-f}$.
Conversely, an automorphism $\alpha$ in this kernel has $h(x)=x^{-1}\alpha(x)\in W$ satisfying $h(xw)=h(x)$ and $h(xy)=h(x)h(y)$.
Hence $h$ factors uniquely through an $\F_p$-linear map $V\to W$.
This kernel identification uses no basis and commutes with isomorphisms, proving naturality.
\end{proof}

Taking orders in the exact sequence of Lemma~\ref{lem:root-automorphisms} gives
\[
|\Aut(G)|=p^{ds}
|\operatorname{Stab}_{\GL(V)\times\GL(W)}(y_G)|.
\]

\subsection{Commutator and power maps}
\label{subsec:parameters}

Let $d,s\geq1$ and $s\leq D_d$, and let
\[
V=\F_p^d,\quad W=\F_p^s,\quad
\Gamma_{p,d,s}=\GL(V)\times\GL(W),\quad
\mathscr B_{p,d,s}=\Hom(\Lambda^2V,W).
\]
Let $\mathscr P_{p,d,s}$ consist of pairs $(\beta,\pi)$ with $\beta\in\mathscr B_{p,d,s}$, where $\pi:V\to W$ is linear for odd $p$, and for $p=2$ satisfies
\[
\pi(0)=0,\quad \pi(v+w)=\pi(v)+\pi(w)+\beta(v,w).
\]
Write $\mathscr P_{p,d,s}^*$ for the pairs satisfying
\[
\operatorname{span}(\beta(V,V)\cup\pi(V))=W.
\]
The change-of-bases action is
\[
(P,Q)\cdot\beta=Q\beta\Lambda^2(P^{-1}),\quad
(P,Q)\cdot\pi=Q\pi P^{-1}.
\]

\begin{proposition}\label{prop:root-parameters}
For $d,s\geq1$ with $s\leq D_d$, the orbits in $\mathscr P_{p,d,s}^*$ correspond exactly to the group classes counted by $B_p(d,s)$, equivalently to the $\GL(d,p)$-orbits on codimension-$s$ subspaces of $W_d$.
The group associated with $(\beta,\pi)$ has induced commutator and $p$th-power maps $\beta$ and $\pi$.
\end{proposition}
\begin{proof}
Fix an ordered basis $v_1,\ldots,v_d$ of $V$.
For odd $p$, the $d$ values $\pi(v_i)$ determine an arbitrary linear map.
For $p=2$, arbitrary values $z_i\in W$ determine exactly one quadratic refinement:
\[
\pi\Bigl(\sum_i t_iv_i\Bigr)
=\sum_i t_iz_i+\sum_{i<j}t_it_j\beta(v_i,v_j).
\]
Expanding at $v+w$ gives the cross terms $\sum_{i<j}(t_iu_j+u_it_j)\beta(v_i,v_j)=\beta(v,w)$.
Conversely, successively adjoining the basis vectors with coefficient one in the polarization identity gives this formula.
Thus each $\beta$ has exactly $p^{ds}$ power maps.
In characteristic two the difference of two refinements is linear, so the fiber over $\beta$ is an affine space with translation space $\Hom(V,W)$.
Substitution in the polarization identity verifies the displayed action, and the span of all transformed values is the image under $Q$ of the original span.
Thus admissibility is preserved.

To realize the parameters, use $U_d$ and $W_d=\Phi(U_d)$ from Lemma~\ref{lem:free-root}.
Identify $U_d/W_d$ with $V$.
A linear map $\theta:W_d\to W$ defines
\[
\beta_\theta(v,w)=\theta([x,y]),\quad
\pi_\theta(v)=\theta(x^p),
\]
where $x,y$ lift $v,w$.
Centrality and exponent $p$ of $W_d$ make the definitions of $\beta_\theta$ and $\pi_\theta$ independent of lifts.
Class-two identities make $\beta_\theta$ bilinear and alternating.
The power identity
\[
(xy)^p=x^py^p[y,x]^{\binom p2}
\]
makes $\pi_\theta$ linear for odd $p$, including scalar compatibility from $(x^a)^p=(x^p)^a$, and quadratic with polarization $\beta_\theta$ for $p=2$.
The coordinate relations
\[
\theta(x_i^p)=\pi_\theta(v_i),\quad
\theta([x_j,x_i])=-\beta_\theta(v_i,v_j)
\]
give a bijection $\Hom(W_d,W)\to\mathscr P_{p,d,s}$; surjectivity of $\theta$ is exactly admissibility.
For surjective $\theta$, let $G_\theta=U_d/\ker\theta$.
Lemma~\ref{lem:free-root} gives the required order, Frattini subgroup and generator number of $G_\theta$.
Conversely, lifts of a Frattini basis in any root give an epimorphism $U_d\to G$ with kernel in $W_d$ of codimension $s$; choosing $W_d/\ker\theta\cong W$ recovers the required $\theta$.

The map $\theta'=Q\theta T_P^{-1}$ has parameters
\[
\beta_{\theta'}(v,w)=Q\beta_\theta(P^{-1}v,P^{-1}w),\quad
\pi_{\theta'}(v)=Q\pi_\theta(P^{-1}v).
\]
Group isomorphisms preserve the commutator and power maps, and compatible pairs lift by Lemma~\ref{lem:root-automorphisms}; hence the $\Gamma_{p,d,s}$-orbits on $\mathscr P_{p,d,s}^*$ correspond exactly to the group isomorphism classes.
Surjections with the same kernel differ by a unique change of basis in $W$, so quotienting first by $\GL(W)$ gives exactly the asserted subspace orbits.

\end{proof}

Define
\[
\begin{aligned}
\Delta_{p,d,s}&=\frac1{|\mathscr B_{p,d,s}|}
 \sum_{\beta\in\mathscr B_{p,d,s}}|\operatorname{Stab}_{\Gamma_{p,d,s}}(\beta)|-(p-1),\\
\eta_{p,d,s}&=\Delta_{p,d,s}+(p-2)p^{-ds},\quad
\sigma_{p,d,s}=\prod_{i=0}^{s-1}(1-p^{i-D_d}).
\end{aligned}
\]

\begin{lemma}\label{lem:parameter-counts}
For $d,s\geq1$ with $s\leq D_d$, $\Delta_{p,d,s},\eta_{p,d,s}\geq0$, $\sigma_{p,d,s}>0$, and
\[
|\mathscr P_{p,d,s}|=p^{sD_d},\quad
|\mathscr P_{p,d,s}^*|=\sigma_{p,d,s}p^{sD_d},\quad
1\leq\mathbb E_{x\in\mathscr P_{p,d,s}}
 |\operatorname{Stab}_{\Gamma_{p,d,s}}(x)|\leq1+\eta_{p,d,s}.
\]
\end{lemma}
\begin{proof}
Choose an ordered basis $v_1,\ldots,v_d$ of $V$.
For each $\beta$, the proof of Proposition~\ref{prop:root-parameters} gives exactly $p^{ds}$ power maps.
The $\binom d2$ values $\beta(v_i,v_j)$ and the $d$ values $\pi(v_i)$ are arbitrary and determine the pair.
Their span equals the span of all commutator and power values, by the preceding formulas.
Therefore admissible pairs correspond to surjective linear maps $\F_p^{D_d}\to W$.
Counting independent rows gives
\[
|\mathscr P^*|=\prod_{i=0}^{s-1}(p^{D_d}-p^i)
=\sigma_{p,d,s}p^{sD_d}.
\]
Every factor of $\sigma$ is positive because $i<s\leq D_d$.

The subgroup
\[
\mathcal Z=\{(\alpha I_V,\alpha^2I_W):\alpha\in\F_p^\times\}
\]
fixes every alternating map.
Expanding stabilizer orders and exchanging finite sums proves
\[
\Delta_{p,d,s}
=\sum_{g\notin\mathcal Z}\Pr_\beta(g\beta=\beta)\geq0.
\]
The alternating marginal of uniform joint parameters is uniform, because the power fibers all have size $p^{ds}$.
Hence the contribution of $g\notin\mathcal Z$ to the joint average is at most $\Delta_{p,d,s}$.
For odd $p$, $(\alpha I,\alpha^2I)$ acts on a linear power map as $\pi\mapsto\alpha\pi$.
For $\alpha\ne1$, the scalar action fixes exactly $\pi=0$, of probability $p^{-ds}$.
The scalar contribution is consequently $1+(p-2)p^{-ds}$.
For $p=2$ the scalar subgroup contains only the identity, giving the same formula.
The identity provides the lower bound one.

\end{proof}

Let
\[
\lambda(m)=\frac{2m^3}{27}-\frac{4m^2}{9}.
\]

\begin{lemma}\label{lem:root-lower-bound}
For every prime $p$ and integer $m\geq1$, $R_p(m)\geq p^{\lambda(m)}$.
More precisely, for $r\geq1$ and $0\leq s\leq D_r$,
\[
B_p(r,s)\geq\frac{\genfrac{[}{]}{0pt}{}{D_r}{s}_p}{|\GL(r,p)|}
\geq p^{s(D_r-s)-r^2},\quad
\genfrac{[}{]}{0pt}{}{D}{s}_p=\prod_{i=0}^{s-1}\frac{p^D-p^i}{p^s-p^i}.
\]
\end{lemma}
\begin{proof}
The subspace-orbit classification in Proposition~\ref{prop:root-parameters} gives the first inequality: there are $\genfrac{[}{]}{0pt}{}{D_r}{s}_p$ subspaces and each orbit has at most $|\GL(r,p)|$ elements.
For $s=0$ the unique subspace gives the elementary abelian group.
Counting bases gives the Gaussian product, whose factors satisfy
\[
\frac{p^{D_r}-p^i}{p^s-p^i}
=p^{D_r-s}\frac{1-p^{i-D_r}}{1-p^{i-s}}\geq p^{D_r-s},
\quad |\GL(r,p)|<p^{r^2}.
\]
For $1\leq m\leq6$, $\lambda(m)=2m^2(m-6)/27\leq0$.
For $m>6$ choose $r=\lfloor2m/3\rfloor$ and $s=m-r$.
For $m=3a,3a+1,3a+2$ these pairs are respectively $(2a,a),(2a,a+1),(2a+1,a+1)$, with $a\geq2$ and $s\leq D_r$.
Writing $E=s(D_r-s)-r^2$, substitution gives
\[
E-\lambda(m)=
\begin{cases}
0&\text{if }m=3a\\
a-17/27&\text{if }m=3a+1\\
2a/3+5/27&\text{if }m=3a+2.
\end{cases}
\]
The three values of $E-\lambda(m)$ are nonnegative, proving $R_p(m)\geq B_p(r,s)\geq p^{\lambda(m)}$.
\end{proof}

\begin{lemma}\label{lem:root-orbit-identity}
Under the hypotheses of Lemma~\ref{lem:parameter-counts}, let $Y=\mathscr P_{p,d,s}$, $Y_{\mathrm{adm}}=\mathscr P_{p,d,s}^*$, $\Gamma=\Gamma_{p,d,s}$, and $r(y)=|\operatorname{Stab}_\Gamma(y)|$.
With expectation over all complete parameters, let
\[
\zeta=\mathbb E_Y[\mathbf1_{Y_{\mathrm{adm}}}(r-1)].
\]
Then
\begin{equation}\label{eq:sharp-orbit-identity}
|\Gamma|B_p(d,s)=(\sigma_{p,d,s}+\zeta)p^{sD_d},\quad
0\leq\zeta\leq\eta_{p,d,s}.
\end{equation}
Admissibility requires commutator and power values jointly to span $W$; it does not require the commutator map alone to be onto.
\end{lemma}
\begin{proof}
The nonnegative quantity $r-1$ has average at most $\eta_{p,d,s}$ on $Y$, so the same holds after multiplication by the admissibility indicator.
Each orbit $O$ contributes $\sum_{y\in O}r(y)=|\Gamma|$.
Summing over admissible orbits gives $|\Gamma|B_p(d,s)=|Y_{\mathrm{adm}}|+\zeta|Y|$, as asserted.
More generally, for every invariant $U\subseteq Y_{\mathrm{adm}}$,
\[
\frac{\#(U/\Gamma)}{B_p(d,s)}
=\frac{\mathbb E_Y[\mathbf1_Ur]}{\sigma_{p,d,s}+\zeta}.\qedhere
\]
\end{proof}

\begin{lemma}\label{lem:root-generator-tail}
Extend $B_p(d,s)$ by its group-theoretic definition to all $d,s\geq0$.
Then
\[
B_p(d,s)\leq p^{sD_d},\quad B_p(d,0)=1,\quad
B_p(0,s)=0\ (s>0),\quad B_p(d,s)=0\ (s>D_d).
\]
For $m\geq1$ and $\rho>0$,
\[
\sum_{\substack{0\leq d\leq m\\|d/m-2/3|\geq\rho}}B_p(d,m-d)
\leq(m+1)p^{(\frac{2}{27}-\frac{\rho^2}{6})m^3+\frac{m^2}{8}}.
\]
For fixed $\rho$, the preceding bound is at most $p^{(2/27-\rho^2/12)m^3}$ for all sufficiently large $m$.
\end{lemma}
\begin{proof}
In the nonempty positive-dimensional range, the number of orbits is at most the number $p^{sD_d}$ of complete parameters.
If $s=0$, the group is elementary abelian, including the trivial group when $d=0$.
If $d=0<s$, a nontrivial finite group would equal its Frattini subgroup, contradicting the existence of a maximal proper subgroup.
If $s>D_d$, the $d$ powers and $\binom d2$ commutators cannot span $\Phi(G)$.
For $x=d/m\in[0,1]$ use
\[
\frac2{27}-\frac{x^2(1-x)}2
=\frac{(x-2/3)^2(x+1/3)}2,\quad
(m-d)D_d=\frac{m^3x^2(1-x)}2+\frac{m^2x(1-x)}2.
\]
The inequalities $x+1/3\geq1/3$ and $x(1-x)\leq1/4$ give the first tail bound after summing at most $m+1$ layers.
Finally, $m+1\leq2^m\leq p^m$ and $m+m^2/8\leq\rho^2m^3/12$ eventually.
\end{proof}

\section{Tensor stabilizers and generator localization}\label{sec:tensor-localization}
\subsection{Stabilizers of alternating tensors}
\label{subsec:tensor-stabilizers}

\begin{proposition}\label{prop:tensor-stabilizers}
Fix $0<a\leq b$.
There are $c>0$ and $n_0$, depending only on $a,b$, such that, for every prime power $q$, $n\geq n_0$, and $an\leq s\leq bn$, the natural action of $\Gamma=\GL(n,q)\times\GL(s,q)$ on $X=\Lambda^2\F_q^n\otimes\F_q^s$ satisfies
\[
q-1\leq\frac1{|X|}\sum_{x\in X}|\Gamma_x|
\leq q-1+q^{-cn^2}.
\]
\end{proposition}

We shall use the following rank estimate for both the first and the higher stabilizer moments.
\begin{lemma}\label{lem:tensor-rank}
Fix $0<a\leq b$ and $T>0$.
Let $V=\F_q^n$, $W=\F_q^s$, $an\leq s\leq bn$, and let $\eta(P,Q)=\rank((\Lambda^2P)\otimes Q-I)$.
There are integers $K,n_0$, depending only on $a,b,T$, such that the following assertions hold simultaneously for $n\geq n_0$.
\begin{enumerate}[label=\textup{(\alph*)},leftmargin=*]
\item\label{item:far-scalar-rank} Let $B_K$ be the set of pairs $(P,Q)\in\GL(V)\times\GL(W)$ for which there exist $\alpha,\beta\in\F_q^\times$ with $\alpha^2\beta=1$ and
\[
\rank(P-\alpha I)\leq K,\quad \rank(Q-\beta I)\leq K.
\]
Every pair outside $B_K$ satisfies $\eta(P,Q)\geq Tn^2$.
\item\label{item:near-scalar-rank} Let
\[
Z=\{(\alpha I,\alpha^{-2}I):\alpha\in\F_q^\times\},
\quad \kappa_0=\min\{a/8,1/16\}.
\]
Then
\[
\eta(P,Q)\geq\kappa_0n^2\quad((P,Q)\in B_K\setminus Z),
\quad |B_K|\leq(q-1)q^{2K(n+s)}.
\]
\end{enumerate}
\end{lemma}

For $q=2$, $Z=\{(I,I)\}$ and
\[
B_K=\{(P,Q):\rank(P-I)\leq K,\ \rank(Q-I)\leq K\},
\quad |B_K|\leq2^{2K(n+s)}.
\]
\begin{proof}
Choose integers $L\geq2,D\geq1$ with
\[
aL,\ L/2,\ aD/L^2,\ D/(2L^2)>4T,
\quad J=L^2,\quad K=J+D.
\]
All subsequent thresholds depend only on the fixed parameters $a,L,D,J,K$.
Write $A=(\Lambda^2P)\otimes Q$.
Regard $V,W$ as $\F_q[t]$-modules through $P,Q$ and take cyclic decompositions over the original field, obtained by Smith reduction of $tI-P,tI-Q$.
A length-$\ell$ block has a cyclic basis $e_1,\ldots,e_\ell$ with $Pe_i=e_{i+1}$ for $i<\ell$.
Call blocks of length at least two nontrivial; their polynomials need not split over the field.

If $P$ has a block of length $\ell\geq L+1$, let $C$ be the other blocks and let
\[
U=\langle e_1,\ldots,e_L\rangle,\quad
B_0=\langle e_{L+1},\ldots,e_{\ell-1}\rangle\oplus C.
\]
Empty intervals denote zero spaces.
Then $\dim B_0=n-L-1$ and
\[
P(B_0)=\langle e_{L+2},\ldots,e_\ell\rangle\oplus C,\quad
V=\langle e_1,\ldots,e_{L+1}\rangle\oplus P(B_0).
\]
A fixed vector in $(U\wedge B_0)\otimes W$ satisfies
\[
\sum_{i=1}^Le_{i+1}\wedge(P\otimes Q)x_i
=\sum_{i=1}^Le_i\wedge x_i,\quad x_i\in B_0\otimes W.
\]
Project first onto $(e_{L+1}\wedge P(B_0))\otimes W$ and then, in descending order, onto $(e_{i+1}\wedge P(B_0))\otimes W$.
Injectivity of $P\otimes Q$ on $B_0\otimes W$ gives all $x_i=0$.
Thus $A-I$ is injective on $(U\wedge B_0)\otimes W$, which need not be invariant, and
\[
\eta(P,Q)\geq L(n-L-1)s\geq aLn^2/2>Tn^2
\]
eventually.
A target block of length at least $L+1$ similarly gives injectivity on $\Lambda^2V\otimes\langle f_1,\ldots,f_L\rangle$: its last coefficient vanishes first, then the rest.
Thus $\eta(P,Q)\geq L\binom n2\geq Ln^2/4>Tn^2$.

We may assume all blocks have length at most $L$.
If $P$ has at least $L$ nontrivial blocks, select $L$ of them, let $U,F$ span their first and second vectors, and let $B$ be the unselected blocks.
Then $PU=F$, $P(B)=B$ and $\dim B\geq n-L^2$.
The subspaces $(U\wedge B)\otimes W$ and $(F\wedge B)\otimes W$ are disjoint, so $\eta(P,Q)\geq L(n-L^2)s>Tn^2$ eventually.
The analogous target argument gives $\eta(P,Q)\geq L\binom n2>Tn^2$ if $Q$ has at least $L$ nontrivial blocks.
Otherwise the nontrivial blocks of $P$ and of $Q$ each have total dimension less than $J$.
The one-dimensional blocks give diagonal subspaces $V_0,W_0$ with dimensions at least $n-J,s-J$.

All eigenvalues below lie in $\F_q^\times$.
If $P|_{V_0}$ has $L+1$ distinct eigenvalues, choose corresponding basis vectors $v_i$.
For every remaining source basis vector $u$ and target basis vector $w$, at most one of the independent tensors $v_i\wedge u\otimes w$ has eigenvalue one.
Hence $\eta(P,Q)\geq L(n-J-L-1)(s-J)>Tn^2$ eventually.
If $Q|_{W_0}$ has $L+1$ distinct eigenvalues, similarly $\eta(P,Q)\geq L\binom{n-J}{2}>Tn^2$.

There remain at most $L$ eigenvalues for each of $P|_{V_0}$ and $Q|_{W_0}$.
Choose maximal multiplicity eigenvalues $\alpha,\beta$, with eigenspace dimensions $r\geq(n-J)/L$, $u\geq(s-J)/L$.
If $\alpha^2\beta\ne1$, restriction to $\Lambda^2V_\alpha\otimes W_\beta$ gives
\[
\eta(P,Q)\geq\binom r2u\geq an^3/(32L^3)>Tn^2
\]
eventually.
Otherwise $\alpha^2\beta=1$.
If at least $D$ basis vectors of $V_0$ have eigenvalue different from $\alpha$, tensor their span with $V_\alpha,W_\beta$ to obtain $\eta(P,Q)\geq Dru\geq D(n-J)(s-J)/L^2>Tn^2$.
If at least $D$ target basis vectors have eigenvalue different from $\beta$, the bound is $\eta(P,Q)\geq D\binom r2\geq Dn^2/(4L^2)>Tn^2$.
For example $n\geq4(J+L)$ ensures $r,r-1\geq3n/(4L)$.
Thus $\eta<Tn^2$ leaves fewer than $D$ exceptional diagonal vectors on each side, giving kernel dimensions at least $n-K,s-K$ after adding the nontrivial blocks.
All thresholds are uniform in $q,s$.

For a near-scalar pair let $E=\ker(P-\alpha I)$.
If $Q\ne\beta I$, choose $(Q-\beta I)w\ne0$.
For $z\in\Lambda^2E$,
\[
(A-I)(z\otimes w)=\alpha^2z\otimes(Q-\beta I)w,
\]
so $\eta(P,Q)\geq\binom{n-K}{2}$.
If $Q=\beta I$ but $P\ne\alpha I$, choose $v\notin E$ and $y=(P-\alpha I)v\ne0$.
Then
\[
(A-I)((u\wedge v)\otimes w)=\alpha^{-1}(u\wedge y)\otimes w
\quad(u\in E,w\in W).
\]
The input map in $u$ is injective; the output kernel $E\cap\langle y\rangle$ has dimension at most one.
Thus $\eta(P,Q)\geq(n-K-1)s$.
Increasing the threshold makes these bounds at least $n^2/8$ and $an^2/2$, hence at least $\kappa_0n^2$.
Compatible scalar pairs are exactly $Z$ and fix the whole tensor space.
Exterior products make these arguments valid in characteristic two.

A rank-at-most-$K$ matrix factors through a $K$-space, giving at most $q^{2nK}$ source choices and $q^{2sK}$ target choices.
Multiplication by the $q-1$ compatible scalar choices proves the size bound.
At $q=2$ the only compatible scalars are both one.
\end{proof}

\begin{lemma}\label{lem:fixed-space-moments}
Let a finite group $\Gamma$ act linearly on a finite $\F_q$-space $X$.
Write $Z$ for the kernel of the action and $\eta(g)=\operatorname{codim}_X\operatorname{Fix}(g)$.
For every integer $h\geq1$ and uniform $x\in X$,
\[
\mathbf E|\Gamma_x|^h
 =\sum_{g_1,\ldots,g_h\in\Gamma}
 q^{-\operatorname{codim}_X(\bigcap_i\operatorname{Fix}(g_i))}.
\]
If $Z\subseteq B\subseteq\Gamma$, $\eta\geq L$ on $\Gamma\setminus B$ and $\eta\geq H$ on $B\setminus Z$, where $L,H\geq0$, then
\[
 |Z|^h\leq\mathbf E|\Gamma_x|^h
 \leq |Z|^h+|\Gamma|^h q^{-L}+|B|^h q^{-H}.
\]
\end{lemma}
\begin{proof}
Expand stabilizer indicators.
Separate $Z^h$, other tuples containing a member with $\eta\geq L$, and the rest, in $B^h\setminus Z^h$.
The three classes of tuples give the three terms in the upper bound because a common fixed space has codimension at least that of any member, without additivity or independence.
\end{proof}

\begin{proof}[\rm\textbf{Proof of Proposition~\ref{prop:tensor-stabilizers}}]
Use Lemma~\ref{lem:tensor-rank} with $M=1+b^2$, $T=M+1$.
The kernel of this tensor action satisfies $|Z|=q-1$, and $|\Gamma|\leq q^{n^2+s^2}\leq q^{Mn^2}$.
Lemma~\ref{lem:fixed-space-moments}, with $h=1$, $B=B_K$, $L=Tn^2$ and $H=\kappa_0n^2$, bounds the nonnegative excess above $q-1$ by $q^{-n^2}+q^{2+2K(n+s)-\kappa_0n^2}$.
The second term is eventually at most $q^{-\kappa_0n^2/2}$.
Increasing the threshold and using $q\geq2$ bounds the sum by $q^{-cn^2}$ with $c=\kappa_0/4$, uniformly in $q,s$.
\end{proof}

For denominator estimates let $a_{p,n}=\prod_{i=1}^n(1-p^{-i})$ and $a_{p,0}=1$.
Induction gives $\prod_i(1-u_i)\geq1-\sum_i u_i$ for $0\leq u_i\leq1$: the last multiplication adds the nonnegative correction $u_r\sum_{i<r}u_i$.
Thus $a_{p,n}\downarrow g_p$ and
\[
 g_p\geq(1-p^{-1})(1-p^{-2})
       \left(1-\frac1{p^2(p-1)}\right)\geq\frac9{32}>\frac14,
 \quad g_p\leq1-p^{-1}.
\]
Expanding the last two factors also gives $g_p\geq(1-p^{-1})(1-1/[p(p-1)])$.
For $N\geq n\geq0$, let $T_{p,n}=g_p/a_{p,n}$ and $G_p=g_p^{-1}$.
The same product inequality and passage to the limit give
\begin{equation}\label{eq:product-bounds}
 \begin{aligned}
 0\leq1-\prod_{i=n+1}^N(1-p^{-i})&\leq\frac{p^{-n}}{p-1},\\
 0\leq1-T_{p,n}&\leq\frac{p^{-n}}{p-1},\quad 0<T_{p,n}\leq1,\quad G_p<4.
 \end{aligned}
\end{equation}
In particular, $|\GL(n,p)|=p^{n^2}a_{p,n}\geq p^{n^2}/4$.

\begin{corollary}\label{cor:parameter-errors}
Given $0<a\leq b$, there are $c>0,d_0$ such that for all primes $p$, $d\geq d_0$ and $ad\leq s\leq bd$,
\[
0\leq\Delta_{p,d,s}\leq p^{-cd^2}.
\]
For $1\leq s\leq D_d$, $0\leq1-\sigma_{p,d,s}\leq p^{s-D_d}/(p-1)$.
\end{corollary}
\begin{proof}
The identification $\Hom(\Lambda^2V,W)\cong(\Lambda^2V^*)\otimes W$ turns the source action into $(P^{-1})^*$, or $P^{-T}$ in coordinates.
The map $P\mapsto P^{-T}$ is a group isomorphism, since $(P_1P_2)^{-T}=P_1^{-T}P_2^{-T}$, and preserves stabilizer orders.
Proposition~\ref{prop:tensor-stabilizers} with $q=p$ gives the first bound.
The preceding product inequality gives
\[
1-\sigma_{p,d,s}\leq\sum_{i=0}^{s-1}p^{i-D_d}
=\frac{p^{s-D_d}-p^{-D_d}}{p-1}\leq\frac{p^{s-D_d}}{p-1}.\qedhere
\]
\end{proof}
\subsection{Localization for all finite \texorpdfstring{$p$}{p}-groups}
\label{subsec:generator-localization}

We use the presentation count in \cite[Chapter~5]{BNV2007}, based on Sims's method \cite{Sims1965} and the alternating-map refinement attributed there to Newman and Seeley \cite[p.~33]{BNV2007}.
Keeping the generator number in these estimates gives the following localization after a parameter-dependent optimization.

\begin{proposition}\label{prop:generator-localization}
Fix $p$.
For each $\varepsilon>0$, the number of order-$p^m$ isomorphism classes with $|d(G)/m-2/3|\geq\varepsilon$ is at most
\[
p^{(\frac{2}{27}-c_\varepsilon)m^3+O(m^{\frac{8}{3}})},\quad
c_\varepsilon=\min\{7/3375,\ 7\varepsilon^2/15\}>0.
\]
The error constant is independent of $\varepsilon$.
\end{proposition}
\begin{proof}
For a nonabelian group $G$ of order $p^m$, let
\[
r_i=d(\gamma_i(G)/\gamma_{i+1}(G)),\quad d=r_1=d(G),\quad b=r_2.
\]
The $r_i$ are minimal generator numbers, not logarithmic orders.
Set
\[
V=G/\Phi(G),\quad W=G'/((G')^p\gamma_3(G)).
\]
Then $\dim V=d$ and $\dim W=b\geq1$.
The commutator induces a surjective linear map $\beta:\Lambda^2V\to W$.
Its underlying bilinear map is alternating in every characteristic.
Let $\sigma$ be the least dimension of $U\leq V$ such that $\beta(\Lambda^2U)=W$.
Corollary~5.3 of \cite[p.~33]{BNV2007} gives
\[
2\leq\sigma\leq\min(d,b+1),\quad
q=\sigma-1,\quad u=m-d-b\geq0,\quad 0\leq q\leq\min(d,b).
\]

Write $N_p(d,b,\sigma)$ for the base-$p$ logarithm of the number of labeled alternating maps $\Lambda^2\F_p^d\to\F_p^b$ with this minimal surjective-subspace dimension.
This is the notation of \cite[Proposition~5.5, pp.~33--34]{BNV2007}, with $(r_1,r_2,s)=(d,b,\sigma)$.
Equations~(5.5) and~(5.6) of \cite{BNV2007} give
\[
\begin{aligned}
N_p(d,b,\sigma)&\leq d^2(b-q)/2+O((d+b)^{\frac{8}{3}}),\\
N_p(d,b,\sigma)&\leq d^2(b-q)/2+dbq/2+O((d+b)^{\frac{5}{2}}).
\end{aligned}
\]
We apply the two estimates only to nonempty strata, so $N_p(d,b,\sigma)$ is defined.

Let $T_p(m;d,b,\sigma)$ count the nonabelian order-$p^m$ isomorphism classes with parameters $(d,b,\sigma)$.
The restricted presentation theorem \cite[Theorem~5.6, p.~38; construction and proof, pp.~35--40]{BNV2007} assigns to every such group structural indices, power data and a specified subset of commutator data which together determine its isomorphism class.
Counting the possible data therefore gives an upper bound on the number of isomorphism classes, even though some choices of data need not define a group.
The proof of Theorem~5.7 bounds all choices of structural indices and power data by $2^{3m-2}p^{m^2}$ \cite[p.~41]{BNV2007}.
Restricting $d,b,\sigma$ can only reduce the number of choices of structural indices and power data.

For each fixed choice of structural indices and power data, the commutator count retains $r_1,r_2,s$ throughout \cite[pp.~42--43, equations~(5.12)--(5.14)]{BNV2007}.
Substituting $(r_1,r_2,s-1,m-r_1-r_2)=(d,b,q,u)$ into (5.14), and absorbing the preceding structural and power factor into the quadratic error, gives, whenever $T_p(m;d,b,\sigma)>0$,
\[
\log_p T_p(m;d,b,\sigma)
\leq N_p(d,b,\sigma)+\frac12d^2u
+\left(qb-\frac12q^2\right)u+\frac12qu^2+O(m^2).
\]
The summation over the nilpotency class, the higher-layer generator numbers, cyclic lengths and intermediate image dimensions is already included in this bound.
For fixed $p$, the error constants in the two estimates for $N_p(d,b,\sigma)$ and the estimate for $\log_p T_p(m;d,b,\sigma)$ are uniform over the admissible parameters; in particular the constants do not depend on $\varepsilon$.

Normalize by $x=d/m$, $y=b/m$, $z=u/m$, $w=q/m$, so $x+y+z=1$, $0\leq w\leq\min(x,y)$.
The leading polynomials are
\[
\begin{aligned}
\mathcal A&=x^2(y+z-w)/2+xyw/2+(wy-w^2/2)z+wz^2/2,\\
\mathcal B&=x^2(y+z-w)/2+(wy-w^2/2)z+wz^2/2.
\end{aligned}
\]
For $x\geq3/5$, eliminating $y$ gives
\[
\partial_w\mathcal A=x/2-x^2+(1-3x/2)z-z^2/2-wz<0.
\]
For $x\geq2/3$, the negativity of $\partial_w\mathcal A$ follows from the signs; for $3/5\leq x\leq2/3$, completion of the square in $z$ bounds $\partial_w\mathcal A$ by $1/2-x+x^2/8\leq-11/200$.
Thus $\mathcal A\leq f(x)=x^2(1-x)/2$.
For $x\leq3/5$,
\[
\mathcal B=f(x)+w(1-2x-y^2)/2-w^2(1-x-y)/2.
\]
When $1/2\leq x\leq3/5$, $\mathcal B$ is at most $f(3/5)=9/125$.
For $x\leq1/2$, $\mathcal B$ decreases with $y\geq w$, so take $y=w$ and optimize the resulting concave quadratic.
If $x\leq1-1/\sqrt2$, the bound is $x/2-x^2\leq1/16$; otherwise it is
\[
f(x)+\frac{(1-2x)^2}{8(1-x)}\leq\frac1{16}.
\]
After multiplication by $16(1-x)$ this is equivalent to $(2x-1)(4x^3-6x^2+5x-1)\leq0$.
The second factor is positive on this interval: its derivative is $12(x-1/2)^2+2$ and its left endpoint value is $5-7/\sqrt2>0$.
Thus the low-generator region has gap $2/27-9/125=7/3375$.
In the high-generator region use $2/27-f(x)=\tfrac12(x-2/3)^2(x+1/3)$ to obtain $7\varepsilon^2/15$.
Sum over polynomially many parameters.
Abelian types correspond to partitions and are at most $2^m$ by comparison with compositions.
Summing over the parameter choices and including the abelian groups preserve the uniform error bound.
\end{proof}
\section{Covers, nuclei and weighted root counts}\label{sec:covers-nuclei}
\subsection{Covers and central extensions}
\label{subsec:covers}

All cohomology in what follows has trivial coefficients $\F_p$; write $r(Y)=\dim H^2(Y,\F_p)$.
For a nontrivial root $Q$, take a minimal free presentation $F/R=Q$ and define
\[
K=F/(R^p[R,F]),\quad M=R/(R^p[R,F]),\quad N=\gamma_3^{(p)}(K),
\quad \nu(Q)=\dim N.
\]
The group $N$ is the nucleus used throughout the proof.
Write $d=d(Q)$ and $s=\dim\Phi(Q)$.

\begin{lemma}\label{lem:root-cover}
The group $K$ is finite, $M$ is central elementary abelian, $M\leq\Phi(K)$, and
\[
\gamma_4^{(p)}(K)=1,\quad K/N\cong U_d,\quad r(Q)=D_d-s+\nu(Q).
\]
\end{lemma}
\begin{proof}
We identify $H^2(Q,\F_p)$ with the space of equivalence classes of central extensions of $Q$ by $\F_p$, with the kernel and quotient fixed.
The zero class corresponds to the split extension, and nonzero scalar multiples give isomorphic middle groups.

The subgroup $R$ is finitely generated.
To see this, let $u_q$ be word representatives for $Q$.
The elements $u_qx_i u_{qq_i}^{-1}$ generate $R$, by inserting representatives at successive prefixes of a relation word and telescoping.
Thus $M$ is finite elementary abelian and central in $K$; since $K/M=Q$, the group $K$ is finite.
The group $K$ is generated by $d$ elements and maps onto the $d$-generated group $Q$, so the induced map $K/\Phi(K)\to Q/\Phi(Q)$ is an isomorphism and $M\leq\Phi(K)$.

For $\theta:M\to\F_p$, form the pushout
\[
E_\theta=(K\times\F_p)/\{(z,-\theta(z)):z\in M\}.
\]
Section cocycles make $\theta\mapsto[E_\theta]$ linear.
If the extension $E_\theta$ is split, projection onto its distinguished kernel gives a homomorphism $K\to\F_p$ restricting to $\theta$ on $M$.
Such a homomorphism kills $\Phi(K)$, hence $\theta=0$.
Conversely, lifts of the chosen generators in any central extension give a map $F\to E$.
The relators map into a central exponent-$p$ kernel, so the map factors through $K$.
The pushout associated with the induced map $M\to\F_p$ has the same finite order as $E$, and the induced map from the pushout to $E$ is the identity on kernel and quotient; hence the two extensions are equivalent.
Consequently $H^2(Q,\F_p)\cong M^*$.
Minimality gives $R\leq \gamma_2^{(p)}(F)$, so $R^p[R,F]\leq \gamma_3^{(p)}(F)$ and $K/\gamma_3^{(p)}(K)\cong F/\gamma_3^{(p)}(F)\cong U_d$.
Since $\gamma_3^{(p)}(Q)=1$, $N\leq M$; thus $N$ is central elementary abelian and $\gamma_4^{(p)}(K)=1$.
Comparing orders gives $\dim M-\dim N=D_d-s$.
\end{proof}

For a finite $p$-group $Y$ with $Y/\gamma_3^{(p)}(Y)\cong Q$, define
\[
L(Y)=\im\bigl(
H^2(Y/\Phi(Y),\F_p)\xrightarrow{\mathrm{inf}}H^2(Y,\F_p)\bigr),
\quad h(Y)=r(Y)-(D_d-s).
\]
Call the classes in $H^2(Y)\setminus L(Y)$ \emph{allowable}.
This set complement is the union of the affine cosets corresponding to nonzero elements of $H^2(Y)/L(Y)$, not a linear complement.

\begin{lemma}\label{lem:allowable-extensions}
A central extension $1\to A\to E\to Y\to1$ with $|A|=p$ represents a class in $L(Y)$ if and only if its kernel $A$ is not contained in $\gamma_3^{(p)}(E)$.
Moreover, $\dim L(Y)=D_d-s$ and $h(Y)\geq0$.
The number of allowable scalar lines is
\[
p^{D_d-s}\frac{p^{h(Y)}-1}{p-1}.
\]
\end{lemma}
\begin{proof}
For the root $Q$, a pushout $E_\theta$ has trivial $\gamma_3^{(p)}$ exactly when $\theta$ kills $N$.
The classes of such pushouts form a subspace $L_0(Q)$ of dimension $D_d-s$.
Inflation to $Y$ is injective on $L_0(Q)$: a splitting of a pullback gives a lifting map $Y\to E$, which kills $\gamma_3^{(p)}(Y)$ because $\gamma_3^{(p)}(E)=1$, and descends to a section of the original extension.
If $X=Y\times_Q E$ is the pullback of an extension $E$ representing a class in $L_0(Q)$, then the projection $X\to E$ kills $\gamma_3^{(p)}(X)$ and is injective on the kernel $A$ of $X\to Y$.
Hence $A\cap\gamma_3^{(p)}(X)=1$.
Conversely, let $1\to A\to X\to Y\to1$ be a central extension with $|A|=p$ and $A\cap\gamma_3^{(p)}(X)=1$.
Projection identifies $\gamma_3^{(p)}(X)$ with $\gamma_3^{(p)}(Y)$, and $X/\gamma_3^{(p)}(X)$ is a $\gamma_3^{(p)}$-trivial central extension of $Q$.
The quotient maps give
\[
X\cong Y\times_Q(X/\gamma_3^{(p)}(X)).
\]
The combined map $X\to Y\times_Q(X/\gamma_3^{(p)}(X))$ is injective by the trivial intersection and onto by equality of orders.
Thus the inflation image of $L_0(Q)$ in $H^2(Y)$ consists exactly of the extension classes for which $A\cap\gamma_3^{(p)}(X)=1$.

Finally $L_0(Q)$ is the inflation image from $V=Q/\Phi(Q)$.
In a $\gamma_3^{(p)}$-trivial central extension with kernel $A$, the group $\Phi(E)A$ is elementary abelian.
Choose a complement to $A$ in $\Phi(E)A$ and quotient $E$ by that complement; the resulting quotient is a scalar extension of $V$, whose pullback is the original extension of $Q$.
Conversely, all scalar central extensions of $V$ have trivial $\gamma_3^{(p)}$.
The quotient map $Y\to Q$ induces an isomorphism $Y/\Phi(Y)\cong V$, so transitivity of inflation identifies the image of $L_0(Q)$ in $H^2(Y)$ with $L(Y)$.
Since $|A|=p$, the condition $A\cap\gamma_3^{(p)}(X)=1$ is equivalent to $A\not\leq\gamma_3^{(p)}(X)$, proving the asserted equivalence.
The injectivity proved above gives $\dim L(Y)=D_d-s$, and hence $h(Y)\geq0$.
The complement of $L(Y)$ has $p^{r(Y)}-p^{D_d-s}$ vectors, and each scalar line has $p-1$ nonzero vectors.
This gives the formula.
\end{proof}

Every descendant can be reached through allowable scalar extensions.
Every nontrivial normal subgroup of a finite $p$-group meets the center nontrivially, by conjugation-orbit counting modulo $p$.
Choose a central order-$p$ subgroup inside the normal subgroup, pass to the quotient, and repeat with the image of the normal subgroup.
Applied to $\gamma_3^{(p)}(G)$ and read in reverse, this filtration gives extensions whose kernels lie in the corresponding $\gamma_3^{(p)}$ at every step.

\subsection{Bounds on the nucleus}
\label{subsec:nucleus}

For a root let $V=Q/\Phi(Q)$, $W=\Phi(Q)$, with complete maps $\beta:\Lambda^2V\to W$ and $q:V\to W$.
The power map is linear for odd $p$ and satisfies $q(x+y)=q(x)+q(y)+\beta(x,y)$ for $p=2$; $\operatorname{span}(\im\beta,q(V))=W$.
Define
\[
J_\beta(x\wedge y\wedge z)
=x\otimes\beta(y,z)+y\otimes\beta(z,x)+z\otimes\beta(x,y),
\quad J_\beta:\Lambda^3V\to V\otimes W.
\]
Let $\mathcal J_\beta=\ker J_\beta^*$, viewed as a space of maps $W\to V^*$, and let $R_\beta$ be the maximum matrix rank of a map in $\mathcal J_\beta$.

\begin{proposition}\label{prop:nucleus-bound}
If $s=\dim W\geq1$, then
\[
\nu(Q)\leq s(s+3)/2,\quad \nu(Q)\leq(s+1)R_\beta.
\]
Moreover $R_\beta=0$ implies $\nu(Q)=0$, and $R_\beta=1$ implies $\nu(Q)\leq2$.
\end{proposition}
\begin{proof}
We first describe the relations imposed on the nucleus.
Since $[\gamma_2^{(p)}(K),K]\leq N$ is central of exponent $p$, the subgroup $K'$ centralizes $\gamma_2^{(p)}(K)$.
Moreover, $[z,x^p]=[z,x]^p=1$ for $z\in \gamma_2^{(p)}(K)$ and $x\in K$.
As $\gamma_2^{(p)}(K)=K^pK'$, it follows that $\gamma_2^{(p)}(K)$ is abelian.

For $v\in V$ and $w\in W$, choose lifts $\widetilde v\in K$ and $\widetilde w\in \gamma_2^{(p)}(K)$, and define
\[
T(v\otimes w)=[\widetilde v,\widetilde w],\quad
S(w)=\widetilde w^{\,p}.
\]
Changing $\widetilde v$ multiplies it by an element of $\gamma_2^{(p)}(K)$, which centralizes $\widetilde w$.
Changing $\widetilde w$ multiplies it by an element of the central elementary abelian subgroup $M$.
Thus both maps are well defined.
The centrality of $[\gamma_2^{(p)}(K),K]$ makes $T$ bilinear, and the fact that $\gamma_2^{(p)}(K)$ is abelian makes $S$ linear.
Since $N=\gamma_2^{(p)}(K)^p[\gamma_2^{(p)}(K),K]$, the images of $T$ and $S$ span $N$.
The Hall--Witt identity and collection in $K$ give
\[
TJ_\beta=0,\quad
T(x,q(y))=S\beta(x,y)-c_pT(y,\beta(x,y)),\quad
c_p=\binom p2\in\F_p.
\]
For the second identity, use $[x,y^p]=[x,y]^p[[x,y],y]^{\binom p2}$ and $[[x,y],y]=[y,[x,y]]^{-1}$.
Thus $c_2=1$, whereas $c_p=0$ for odd $p$.
Let $\mathcal A(\beta,q)$ be the space of pairs $(a,\delta)\in(V\otimes W)^*\oplus W^*$ satisfying
\[
aJ_\beta=0,\quad
a(x,q(y))=\delta(\beta(x,y))-c_pa(y,\beta(x,y)).
\]
There is an injection $N^*\to\mathcal A(\beta,q)$, $\theta\mapsto(\theta T,\theta S)$, since the images of $T$ and $S$ together span $N$.
In particular $a(x,q(x))=0$.
Thus $\mathcal J_\beta$ imposes only the Jacobi constraint; the power constraint is also required to bound $\nu(Q)$.

Let $B=\im\beta$, $b=\dim B$, $k=s-b$.
Consider the restriction map
\[
\mathcal A(\beta,q)\longrightarrow(V\otimes B)^*,\quad
(a,\delta)\longmapsto a|_{V\otimes B}.
\]
Its image lies in the kernel of $J_{\beta:\Lambda^2V\to B}^*$, and hence has dimension at most $\kappa_B=\dim\operatorname{coker}J_{\beta:\Lambda^2V\to B}$.
In the restriction kernel, $a$ factors through $V\times C$, $C=W/B$, while $q_0:V\to C$ is linear and onto, also for $p=2$.
Compatibility reduces to $a_0(x,q_0(y))=\delta(\beta(x,y))$.
The bilinear form $(x,y)\mapsto a_0(x,q_0(y))$ is alternating and vanishes when either input lies in $\ker q_0$, so it descends to an alternating form on $C$, with dimension at most $\binom k2$ of choices.
Once this alternating form is fixed, $\delta$ has at most $k=\dim B^\perp$ degrees of freedom.
Hence
\[
\nu(Q)\leq\dim\mathcal A(\beta,q)
\leq\kappa_B+k(k+1)/2.
\]

For a $b$-space $E\leq\Lambda^2V^*$, row-reduce a basis to distinct leading exterior monomials.
Each leading pair has $d-2$ degree-three extensions, and two pairs share at most one triple.
Distinct leading terms are independent, also in characteristic two.
Thus exterior multiplication $V^*\otimes E\to\Lambda^3V^*$ has rank at least $b(d-2)-\binom b2$.
Duality yields $\kappa_B\leq b(b+3)/2$; substituting this bound into $\nu(Q)\leq\kappa_B+k(k+1)/2$ gives $\nu(Q)\leq s(s+3)/2$.

For a sharper rank bound, choose a minimum commutator-spanning subspace $U$ for the onto map to $B$, of dimension $\sigma$.
Every hyperplane $\ker\ell<U$ has proper exterior image, so a nonzero target functional gives a nonzero alternating form vanishing on that hyperplane, necessarily of the form $\ell\wedge a$.
A projective target-functional line covers at most the $p+1$ lines in its two-dimensional support.
Hence $p^\sigma-1\leq(p+1)(p^b-1)$, forcing $\sigma\leq b+1$.
Restriction of Jacobi tensors to $U\times B$ is injective: if the restriction of a Jacobi tensor $T$ vanishes, Jacobi on $v,u_1,u_2$ gives $T(v,\beta(u_1,u_2))=0$, and the values $\beta(u_1,u_2)$ with $u_1,u_2\in U$ span $B$.

We use the following matrix-space bound over the original finite field: if all matrices in an $a\times b$ linear space $\mathcal S$ have rank at most $R$, then $\dim\mathcal S\leq R\max(a,b)$.
Transpose if needed to make $a\geq b$.
If $R\geq b$, this follows from $\dim\mathcal S\leq ab$, including $b=0$.
Otherwise take $E=\F_{p^a}$ and a $b$-dimensional $\F_p$-subspace $U\leq E$.
The linearized polynomials
\[
f(Z)=\sum_{i=0}^{b-R-1}c_iZ^{p^i},\quad c_i\in E,
\]
restrict to linear maps $U\to E$.
For nonzero coefficients the polynomial root bound gives kernel dimension at most $b-R-1$, hence rank at least $R+1$.
In particular evaluation is injective on the coefficient space.
The image $\mathcal T$ of this evaluation map has dimension $a(b-R)$ and intersects $\mathcal S$ trivially.
Therefore $\dim\mathcal S+a(b-R)\leq ab$, proving the assertion.
Applying the matrix-space bound to the injective restriction of Jacobi tensors to the commutator-spanning subspace $U\times B$ gives $\kappa_B\leq(b+1)R_\beta$.
Also $V^*\otimes B^\perp\leq\mathcal J_\beta$ and $k\leq d$ by surjectivity of $q_0$, so $R_\beta\geq k$.
Consequently
\[
\nu(Q)\leq(b+1)R_\beta+k(k+1)/2\leq(s+1)R_\beta.
\]

It remains to sharpen the bound when $R_\beta=1$.
If $\beta$ is not onto, then $k=1$; write $B^\perp=\langle f\rangle$.
For $d\geq2$, a rank-one tensor with right factor outside this line, combined with a tensor in $V^*\otimes f$ having independent left factor, would have rank two.
Thus $\mathcal J_\beta=V^*\otimes f$.
For $d=1$, $\beta=0$ and surjectivity of $q_0$ gives $s=1$, with the same conclusion.
The functional $fq$ is nonzero and linear.
The diagonal power constraint on $\ell\otimes f$ reads $\ell(x)(fq)(x)=0$ for every $x$, forcing $\ell=0$.
The remaining power functional has at most one degree of freedom, so $\nu\leq1$.
If $\beta$ is onto, take $0\ne\ell\otimes f\in\mathcal J_\beta$.
Jacobi gives $f\beta=\ell\wedge a\ne0$, and $a\otimes f$ also lies in the kernel.
Any independent right factor would again produce rank two, so $\mathcal J_\beta=\langle\ell,a\rangle\otimes f$ has dimension two.
The compatibility projection has kernel $B^\perp=0$, hence $\nu\leq2$.
\end{proof}

\subsection{Weighted counts of roots}
\label{subsec:weighted-roots}

Let $\mathcal F(d,z)=zD_d-d^2-z^2$.

\begin{proposition}\label{prop:weighted-roots}
For every prime $p$, suppose $d\geq1000$, $s\geq3$, $t\geq1$, and
\begin{equation}\label{eq:root-weight-strip}
m=d+s+t,\quad z=s+t,\quad
\delta=m-3d/2,\quad |\delta|\leq d/1000.
\end{equation}
Then
\[
\sum_{\substack{Q\in\mathcal R_p(d,s)\\\nu(Q)>0}}
p^{t(D_d-s)+t\nu(Q)}
\leq p^{\mathcal F(d,z)-\frac{d^2}{20}}.
\]
\end{proposition}
\begin{proof}
For a rank-$R$ functional $T:V\otimes W\to\F_p$, choose coordinates so $T=\sum_{i=1}^R\ell_i\otimes f_i$.
If $\beta=\sum_i\omega_iw_i$, then $TJ_\beta=\sum_{i=1}^R\ell_i\wedge\omega_i$.
The image of the linear map $\beta\mapsto TJ_\beta$ is spanned by exactly the exterior degree-three monomials using at least one of the first $R$ coordinates.
The constraint therefore has codimension
\[
\Delta_R=\binom d3-\binom{d-R}3.
\]
Fix the rank-$R$ witness $T$ and let $S=\operatorname{Stab}_\Gamma(T)$ for $\Gamma=\GL(V)\times\GL(W)$.
Let $Y_T$ be all complete formal pairs $(\beta,q)$ satisfying $TJ_\beta=0$, without generation or surjectivity conditions.
The number of pairs in $Y_T$ is $p^{sD_d-\Delta_R}$: there are $s\binom d2-\Delta_R$ alternating coordinates and $ds$ power coordinates, the latter being quadratic refinements at $p=2$.
If $n_R$ counts roots of maximum Jacobi rank exactly $R$, mark each with a rank-$R$ witness.
The group $\Gamma$ is transitive on rank-$R$ witnesses, so every marked witness may be moved to $T$; distinct root classes cannot then become one $S$-orbit.
The resulting bounds, proved below, are
\begin{equation}\label{eq:root-rank-bounds}
\begin{aligned}
n_R&\leq p^{sD_d-\Delta_R},\\
n_R&\leq(d+1)p^{sD_d-\Delta_R-\binom{d-R}{2}}
&& (s-R\geq1),\\
n_R&\leq p^{sD_d-\Delta_R-(d-R)^2+3(d-R)}
&& (s-R\geq2).
\end{aligned}
\end{equation}
The first follows from $n_R\leq\#(Y_T/S)\leq|Y_T|$.

Let the left and right radicals of $T$ be $A\leq V$, $W_0\leq W$, with dimensions $n=d-R$, $s-R$.
Extend $H=\GL(A)$ trivially across a fixed complement and on $W$, so $H\leq S$.
Jacobi on $v,a,b$ gives $\beta(\Lambda^2A)\subseteq W_0$.
Every alternating map $\Lambda^2A\to W_0$ extends by zero in the complement directions and still satisfies $TJ_\beta=0$.
For $k=1$ or $2$ with $k\leq s-R$, restrict to $\Lambda^2A$ and project onto fixed $k$ output coordinates in $W_0$.
This is an $H$-equivariant surjection
\[
Y_T\longrightarrow\Hom(\Lambda^2A,\F_p^k)
\]
with fibers of size $p^{sD_d-\Delta_R-k\binom n2}$.
Each source orbit meets the fiber over a chosen target-orbit representative, whence
\[
n_R\leq p^{sD_d-\Delta_R-k\binom n2}
\#\bigl(\Hom(\Lambda^2A,\F_p^k)/H\bigr).
\]
This uses no freeness assumption.
For $k=1$, successively splitting off nondegenerate alternating planes classifies forms by even rank, giving at most $d+1$ orbits and the second bound.

For $k=2$ we use the canonical decomposition of ordered alternating pairs over any field $\mathbb k$.
Write
\[
\mathsf H(P)=\begin{pmatrix}0&P\\-P^{\mathsf T}&0\end{pmatrix}
\quad(P\in\operatorname{Mat}_{a\times b}(\mathbb k)).
\]
Every pair is, under simultaneous congruence $(A_0,A_1)\mapsto(S^{\mathsf T}A_0S,S^{\mathsf T}A_1S)$, a direct sum of the following blocks:
\begin{enumerate}[label=\textup{(\alph*)},leftmargin=*]
\item\label{item:alternating-singular} Singular blocks $(\mathsf H(L_0),\mathsf H(L_1))$ of dimension $2\ell+1$, $\ell\geq0$, where $L_0=\binom{I_\ell}{0}$ and $L_1=\binom{0}{I_\ell}$ are $(\ell+1)\times\ell$ matrices.
At $\ell=0$ this is a one-dimensional zero pair.
\item\label{item:alternating-finite} Finite regular blocks $(\mathsf H(C_g),\mathsf H(I_h))$ of dimension $2h$, where $g=f^e$ is a power of a monic irreducible polynomial, $h=\deg g$, and $C_g$ is its companion matrix.
\item\label{item:alternating-infinity} Infinity regular blocks $(\mathsf H(I_h),\mathsf H(J_h(0)))$ of dimension $2h$, $h\geq1$.
\end{enumerate}
In odd characteristic this is \cite[Corollary~2.4, p.~8]{Sergeichuk}.
For characteristic two, \cite[Proposition~4 and pp.~279--280]{Waterhouse} give the regular and infinity decompositions; the alternating case is treated after Corollary~7 on p.~281, and Proposition~8(iii) and Theorem~9 on p.~282 give the singular part, whose additional diagonal constants are all zero.
Transposing the rectangular blocks if necessary gives the displayed list.
Both forms are alternating, including in characteristic two where their diagonals are zero.
Common radicals, degeneracy and dependent pairs are allowed; there is no field-size restriction relative to dimension and no $\GL_2$ mixing of the ordered outputs.
For the upper bound below, we only need the existence of such a decomposition.

Over $\F_p$, let $c_h$ count candidate block labels of dimension $h$.
For odd $h$ it is one.
For $h=2k$, finite labels are among the $p^k$ monic degree-$k$ polynomials, each already determining its irreducible factor and exponent; there is one infinity label.
Thus $c_{2k}\leq p^k+1\leq p^{2k}$ and $c_h\leq p^h$.
For $n\geq1$, each ordered composition of $n$ has at most $p^n$ ordered block lists, and there are $2^{n-1}$ compositions.
Every simultaneous-congruence orbit of ordered alternating pairs has at least one such block list, so the number of these orbits is at most $2^{n-1}p^n\leq p^{2n}$.
For $n=0$ the empty pair is the unique orbit.
Combining the fiber size $p^{sD_d-\Delta_R-2\binom n2}$ with the bound $p^{2n}$ on target orbits gives the third bound in~\eqref{eq:root-rank-bounds}.

Only $1\leq R\leq s$ need be considered.
By Proposition~\ref{prop:nucleus-bound}, $\nu\leq(s+1)R$, $\nu\leq s(s+3)/2$, and $\nu\leq2$ at $R=1$.
The parameter strip gives $s,t\leq51d/100$.
Direct expansion yields
\[
\begin{aligned}
\frac{\Delta_R-t(s+1)R}{R}
={}&d^2/2-dR/2+R^2/6-ds/2+s^2\\
&-3d/2+R/2+s+1/3-\delta(s+1).
\end{aligned}
\]
The quadratic part of this expansion is at least
\[
d^2/2-ds+7s^2/6
=\tfrac76(s-3d/7)^2+2d^2/7.
\]
Using $\delta_+(s+1)\leq d^2/100$ and $d\geq1000$ gives
\[
\Delta_R-t(s+1)R\geq27Rd^2/100.
\]
For $R\geq5$, the logarithm of the contribution from roots of maximum Jacobi rank $R$, minus $\mathcal F(d,z)$, is at most
\begin{equation}\label{eq:root-rank-margin}
d^2+z^2-st-27Rd^2/100
\leq(1.2601-0.27R)d^2\leq-899d^2/10000.
\end{equation}
For $R\geq2$ with $s-R\geq2$, the two-form bound gives a difference at most
\[
z^2-st-27Rd^2/100+2dR-R^2+3d-3R
\leq-2599d^2/10000,
\]
using $2dR+3d\leq Rd^2/100$.
For $R=1$, use $\Delta_1=\binom{d-1}{2}$ and $\nu\leq2$ to obtain $-2399d^2/10000+188d/25$.
The remaining case $R=2,s=3$ uses the one-form bound, $\nu\leq9$, and
\[
\Delta_2+\binom{d-2}{2}=3d^2/2-13d/2+7,
\]
giving $-2399d^2/10000+1209d/100$.
The residual cases $R=3,4$ have $s\leq5$; use $\nu\leq20$ and $\Delta_R\geq3d^2/2-12d$ to obtain $-2399d^2/10000+111d/5$.
It follows that, for $d\geq1000$, the contribution of each rank stratum is at most $p^{\mathcal F(d,z)-d^2/16}$.
For $R\geq5$, we shall also use the stronger bound~\eqref{eq:root-rank-margin}.
There are at most $d$ strata, and $\log_p d\leq d\leq d^2/80$.
Summing retains saving $d^2/20$.
\end{proof}
\section{Counting descendants}\label{sec:descendants}
\subsection{Cohomology of a central extension}
\label{subsec:central-cohomology}

\begin{lemma}\label{lem:central-cohomology}
For a nonsplit central extension $1\to\mathrm C_p\to E\to Y\to1$ with nonzero class $\xi\in H^2(Y)$, let
\[
a_Y(\xi)=\dim\ker(H^1(Y)\xrightarrow{\ \cup\xi\ }H^3(Y)),
\]
and let $\epsilon_Y(\xi)=1$ if the coefficient Bockstein $\operatorname{Bock}_Y(\xi)$ belongs to $H^1(Y)\cup\xi$, and zero otherwise.
Then
\[
r(E)=r(Y)-1+a_Y(\xi)+\epsilon_Y(\xi).
\]
For an allowable extension with fixed canonical root $Q$, the same formula holds with $r$ replaced by $h$.
\end{lemma}
\begin{proof}
Centrality gives the first-quadrant Lyndon--Hochschild--Serre sequence
\[
E_2^{a,b}=H^a(Y)\otimes H^b(\mathrm C_p)\Longrightarrow H^{a+b}(E),
\quad d_r:E_r^{a,b}\to E_r^{a+r,b-r+1}.
\]
Its limiting terms are the successive quotients of a finite filtration of each $H^n(E)$; the inflation image is $E_\infty^{n,0}$.
For compatible generators $\zeta\in H^1(\mathrm C_p)$ and $\tau\in H^2(\mathrm C_p)$, the required differentials are
\[
d_2(\zeta)=\xi,\quad d_2(\tau)=0,\quad
d_3(\tau)=\operatorname{Bock}_Y(\xi)
\pmod{H^1(Y)\cup\xi}.
\]
The Bockstein comes from $0\to\F_p\to\mathbf Z/p^2\mathbf Z\to\F_p\to0$; an overall sign convention does not affect the argument.
At $p=2$, $\tau=\zeta^2$ and the formulas still hold.
These low-degree differentials for a central cyclic kernel are given in \cite[Introduction]{Leary1993}.

First-quadrant bidegrees and multiplicativity, with $d_2(\chi\zeta)=\pm\chi\cup\xi$, give the limiting terms in total degree two:
\[
\begin{aligned}
E_\infty^{2,0}&=H^2(Y)/\langle\xi\rangle,\\
E_\infty^{1,1}&\cong\ker(H^1(Y)\xrightarrow{\ \cup\xi\ }H^3(Y)),\\
E_\infty^{0,2}&=\ker\bigl(\F_p\tau\xrightarrow{d_3}
H^3(Y)/(H^1(Y)\cup\xi)\bigr).
\end{aligned}
\]
The terms in bidegrees $(2,0)$, $(1,1)$ and $(0,2)$ stabilize on pages three, three and four, respectively, with dimensions $r(Y)-1,a_Y(\xi),\epsilon_Y(\xi)$.
Adding the filtration dimensions proves the formula.
Since $\xi\ne0$, the degree-one calculation also gives $H^1(E)\cong H^1(Y)$, so the generator number is unchanged.
For allowable extensions the root and therefore $d,s$ remain fixed; subtracting $D_d-s$ proves the assertion about $h$.

For later use, the only differentials entering $(3,0)$ are $d_2^{1,1}$ and $d_3^{0,2}$, of ranks $d-a_Y(\xi)$ and $1-\epsilon_Y(\xi)$.
There are no outgoing differentials from bidegree $(3,0)$.
Consequently every such nonsplit central extension satisfies
\[
\dim\ker(H^3(Y)\xrightarrow{\mathrm{inf}}H^3(E))
=d-a_Y(\xi)+1-\epsilon_Y(\xi)=d+r(Y)-r(E).
\]
Allowability is not needed for this last identity.
\end{proof}

\subsection{Bounds on the cohomological dimension and inflation}
\label{subsec:capacity-inflation}

\begin{lemma}\label{lem:capacity-inflation}
The following bounds hold.
\begin{enumerate}[label=\textup{(\alph*)},leftmargin=*]
\item\label{item:state-capacity} Let $Y$ be a descendant of a root $Q$, and let $d=d(Q)$, $s=\dim\Phi(Q)\geq1$ and $j=\log_p|\gamma_3^{(p)}(Y)|$.
Then
\[
0\leq h(Y)\leq s(s+3)/2+2sj.
\]
\item\label{item:inflation-loss} For any finite $p$-groups $X,Y$ with $d(X)=d(Y)=d$, let $X\to Y$ be an epimorphism whose kernel lies in $\Phi(X)$ and has order $p^j$.
Then
\[
\dim\ker(H^3(Y)\xrightarrow{\mathrm{inf}}H^3(X))
\leq jd+r(Y)-r(X).
\]
\end{enumerate}
\end{lemma}
\begin{proof}
We first prove a subgroup criterion.
If $H\leq G$ and $\Phi(H)\gamma_3^{(p)}(G)=\Phi(G)$, then
\[
\Phi(H)=\Phi(G),\quad H\trianglelefteq G,\quad \gamma_3^{(p)}(H)=\gamma_3^{(p)}(G).
\]
The series $\gamma_i^{(p)}(G)$ terminates for every finite $p$-group: by induction on order, quotient by a central order-$p$ subgroup $A$; eventually $\gamma_n^{(p)}(G)\leq A$, and then $\gamma_{n+1}^{(p)}(G)=1$.

First suppose $M\trianglelefteq G$ and $\Phi(M)=\Phi(G)$.
The group $\gamma_3^{(p)}(M)$ is characteristic in $M$, hence normal in $G$.
Since $[M,G]\leq\Phi(G)=\Phi(M)$, conjugation acts trivially on $M/\Phi(M)$ and acts on $M/\gamma_3^{(p)}(M)$ by multiplying elements by central exponent-$p$ elements.
The conjugation action of $G$ therefore fixes all powers and commutators generating $\Phi(M)/\gamma_3^{(p)}(M)$.
Thus $[\Phi(M),G]\leq \gamma_3^{(p)}(M)$ and $\gamma_3^{(p)}(G)\leq \gamma_3^{(p)}(M)\leq \gamma_3^{(p)}(G)$.
For general $H$, induct on $[G:H]$.
If $H<G$, choose a maximal subgroup $M$ containing it.
Then $M\trianglelefteq G$ and $\Phi(M)\trianglelefteq G$.
The hypothesis gives $\Phi(M)\gamma_3^{(p)}(G)=\Phi(G)$.
In $G/\Phi(M)$ the second and third $\gamma^{(p)}$-terms coincide; recursion and termination force them to be trivial.
Thus $\Phi(M)=\Phi(G)$ and the normal case gives $\gamma_3^{(p)}(M)=\gamma_3^{(p)}(G)$.
Apply induction inside $M$ to conclude $\Phi(H)=\Phi(M)=\Phi(G)$.
Now $H$ contains $\Phi(G)$, so it is normal as the inverse image of a subgroup of the abelian quotient $G/\Phi(G)$.
The normal case proves $\gamma_3^{(p)}(H)=\gamma_3^{(p)}(G)$.

Suppose $\Phi(H)=\Phi(Y)$ and let $\widetilde H$ be the full inverse image of $H$ in a central scalar extension $\widetilde Y\to Y$.
For an allowable extension with kernel $A\leq \gamma_3^{(p)}(\widetilde Y)$, $\Phi(\widetilde H)A=\Phi(\widetilde Y)$.
The criterion gives $\gamma_3^{(p)}(\widetilde H)=\gamma_3^{(p)}(\widetilde Y)$, so restriction is allowable.
The converse follows from subgroup inclusion.
Hence $\operatorname{res}^{-1}L(H)=L(Y)$.
The inclusion $H/\Phi(Y)\leq Y/\Phi(Y)$ has a linear retraction.
The restriction map $H^2(Y/\Phi(Y))\to H^2(H/\Phi(Y))$ is therefore onto; since $L(Y)$ and $L(H)$ are the inflation images from these Frattini quotients, $\operatorname{res}L(Y)=L(H)$.
We obtain an injection
\[
H^2(Y)/L(Y)\hookrightarrow H^2(H)/L(H).
\]

To bound capacity, choose a basis of $W=\Phi(Q)$ from commutator and power values.
The inputs to the chosen commutator and power values use at most $2s$ vectors, not necessarily independent.
Let $U$ be the span of these input vectors in $Q/W$, choose a basis of $U$, and let $H_0\leq Q$ be generated by lifts of that basis.
Let $k=\dim U\leq\min(d,2s)$.
Powers and commutators depend only on inputs modulo the central exponent-$p$ subgroup $W$, so their chosen basis values lie in $\Phi(H_0)$.
Thus
\[
\Phi(H_0)=W,\quad H_0/\Phi(H_0)\cong U,\quad d(H_0)=k.
\]
For the full inverse image $H$ of $H_0$ in $Y$, $\Phi(H)\gamma_3^{(p)}(Y)=\Phi(Y)$.
The criterion gives $\Phi(H)=\Phi(Y)$ and $\gamma_3^{(p)}(H)=\gamma_3^{(p)}(Y)$; hence $H/\gamma_3^{(p)}(H)\cong H_0$, $d(H)=k$, and $|\gamma_3^{(p)}(H)|=p^j$.
The restriction injection yields $h(Y)\leq r(H)-(D_k-s)$.
A central order-$p$ filtration of $\gamma_3^{(p)}(H)$ gives exactly $j$ allowable scalar extensions from $H_0$ to $H$, all with generator number $k$.
Lemma~\ref{lem:central-cohomology} bounds each increment in $r$ by $k$, since $-1+a+\epsilon\leq k$.
Therefore
\[
h(Y)\leq r(H)-(D_k-s)\leq\nu(H_0)+kj
\leq s(s+3)/2+2sj,
\]
by Lemma~\ref{lem:root-cover} and Proposition~\ref{prop:nucleus-bound} applied to the actual root $H_0$.
The commutator map of $H_0$ need not be onto.

For inflation, the normal-kernel filtration constructed after Lemma~\ref{lem:allowable-extensions} factors $X\to Y$ into $j$ central scalar extensions $X_0=Y,\ldots,X_j=X$.
Since the original kernel lies in $\Phi(X)$, each scalar kernel lies in the corresponding Frattini subgroup; every step is nonsplit and has generator number $d$.
The one-step degree-three equality in the proof of Lemma~\ref{lem:central-cohomology} and subadditivity of kernel dimensions under composition give
\[
\dim\ker(H^3(Y)\to H^3(X))
\leq\sum_{i=0}^{j-1}(d+r(X_i)-r(X_{i+1}))
=jd+r(Y)-r(X).
\]
For $j=0$, both sides are zero.
\end{proof}

\subsection{Annihilators of extension classes}
\label{subsec:annihilator-tails}

Use the same definition of $a_Y(\xi)$ for all $\xi\in H^2(Y)$, including $a_Y(0)=d(Y)$.
Probabilities in this subsection concern cohomology classes or scalar lines.
These probabilities will be applied conditionally at each vertex of an extension tree.

For a nontrivial finite $p$-group $Y$, set
\[
\begin{gathered}
d=d(Y),\quad s=\dim\Phi(Y/\gamma_3^{(p)}(Y)),\quad j=\log_p|\gamma_3^{(p)}(Y)|,
\\
h=r(Y)-(D_d-s),\quad K_Y=d(s+j)+s-h,
\end{gathered}
\]
and write $G_p=\prod_{i\geq1}(1-p^{-i})^{-1}<4$.
For an affine coset $\mathcal C$ of $L(Y)$, let $\Pr_{\mathcal C}$ denote probability when $\xi$ is chosen uniformly from $\mathcal C$.
When the set of allowable scalar lines is nonempty, let $\Pr_{\mathrm{adm}}$ denote probability for a uniformly chosen line.
In events involving $a_Y(\xi)$, such a line is represented by any nonzero $\xi$ on it; $a_Y(\xi)$ is unchanged by nonzero scaling.

\begin{proposition}\label{prop:annihilator-tail}
We have $K_Y\geq0$.
For every affine coset $\mathcal C$ of $L(Y)$ and $3\leq r\leq a\leq d$,
\[
\Pr_{\mathcal C}(a_Y(\xi)\geq a)
\leq\min\left\{1,G_p p^{r(d-a)-\binom d2+\frac{3K_Y}{r-2}}\right\}.
\]
The same bound holds with $\Pr_{\mathcal C}$ replaced by $\Pr_{\mathrm{adm}}$ whenever the allowable-line set is nonempty.
\end{proposition}
\begin{proof}
Apply Lemma~\ref{lem:capacity-inflation} to $Y\to V=Y/\Phi(Y)$, whose kernel has order $p^{s+j}$ and whose target has $\dim H^2(V)=D_d$.
Then $Z_Y=\ker(H^3(V)\to H^3(Y))$ satisfies $0\leq\dim Z_Y\leq d(s+j)+D_d-r(Y)=K_Y$, also proving $K_Y\geq0$.
The elementary abelian cohomology algebra is
\[
H^*(V)=
\begin{cases}
\Lambda(V^*)\otimes\operatorname{Sym}(\operatorname{Bock}V^*)&\text{if }p\text{ is odd}\\
\operatorname{Sym}(V^*)&\text{if }p=2.
\end{cases}
\]
The exterior generators have degree one and the polynomial generators in the odd case have degree two.
This calculation follows from periodic resolutions for cyclic factors and their tensor product.
Take the quotient by the degree-two polynomial generators for odd $p$, and by squares for $p=2$.
In either characteristic, the quotient gives an algebra map $\pi:H^*(V)\to\Lambda(V^*)$.
Hence $\mathscr Z=\pi(Z_Y)\leq\Lambda^3V^*$ has dimension at most $K_Y$.

We need the elementary exterior-shadow inequality
\[
(r-2)\dim F\leq3\dim(A\wedge F)
\quad(A\leq T,\ \dim A=r>2,\ F\leq\Lambda^2T).
\]
Choose a basis whose first $r$ vectors span $A$, order supports by distinct powers of two, and row-reduce $F$.
Each leading pair has at least $r-2$ extensions, while a triple arises from at most three pairs.
Distinct resulting leading terms are independent, proving the inequality.

Uniform $z\in H^2(V)$ gives a uniform class $\xi=\xi_0+\operatorname{inf}z$ in any fixed coset of $L(Y)$, because inflation onto $L(Y)$ has equal fibers.
Choose $H^2(V)=E\oplus P$, where $\dim E=\binom d2$ and $\pi:E\to\Lambda^2V^*$ is an isomorphism; $P$ is the square or Bockstein part.
Fix the $P$-coordinate of $z$.
For an $r$-space $A\leq H^1(Y)$, the requirement $A\cup\xi=0$ is an affine linear system on $E$.
If the solution set is nonempty, its difference space $F_0$ satisfies $A\wedge\pi(F_0)\leq\mathscr Z$.
The exterior-shadow inequality gives $\dim F_0\leq3K_Y/(r-2)$, so the probability of $A\cup\xi=0$ is at most $p^{-\binom d2+3K_Y/(r-2)}$.

An annihilator of dimension $a_Y(\xi)$ contains $\genfrac{[}{]}{0pt}{}{a_Y(\xi)}{r}_p$ subspaces of dimension $r$.
Double counting therefore gives
\[
\mathbb E\genfrac{[}{]}{0pt}{}{a_Y(\xi)}{r}_p
\leq\genfrac{[}{]}{0pt}{}{d}{r}_p p^{-\binom d2+\frac{3K_Y}{r-2}}.
\]
On $a_Y(\xi)\geq a$ its left-hand random variable is at least $\genfrac{[}{]}{0pt}{}{a}{r}_p$, and the Gaussian product formula gives $\genfrac{[}{]}{0pt}{}{d}{r}_p/\genfrac{[}{]}{0pt}{}{a}{r}_p\leq G_pp^{r(d-a)}$.
Equation~\eqref{eq:product-bounds} gives $G_p<4$.
Allowable classes are a union of equal-sized nonzero affine cosets; each allowable line has $p-1$ nonzero elements.
Since scaling does not change annihilators, the same bound holds on the nonempty allowable-line set.
\end{proof}

\begin{corollary}\label{cor:annihilator-estimates}
Let $\Pr$ denote $\Pr_{\mathcal C}$ for a fixed affine coset $\mathcal C$ of $L(Y)$, or $\Pr_{\mathrm{adm}}$ when the allowable-line set is nonempty.
Suppose $d\geq10000$ and $K_Y\leq3d^2/5$, and let $q_0=\lfloor d/100\rfloor$ and $\alpha=d/2-1000$.
Then
\[
\Pr(a_Y(\xi)\geq a)\leq p^{-\alpha a}\quad(q_0\leq a\leq d),
\quad \Pr(a_Y(\xi)\geq100)\leq p^{-\frac{9d^2}{20}}.
\]
\end{corollary}
\begin{proof}
Apply Proposition~\ref{prop:annihilator-tail} with $r=q_0$, using $q_0-2\geq d/200$.
The base-$p$ logarithm of the probability is at most $2+(d/100)(d-a)-d(d-1)/2+360d$.
After adding $\alpha a$, this becomes
\[
2-49d(d-a)/100+721d/2-1000a<0:
\]
it increases with $a$ and is already negative at $a=d$.
Finally take $r=a=100$.
Since $3K_Y/98<d^2/50$, the base-$p$ logarithm of $\Pr(a_Y(\xi)\geq100)$ is at most $-12d^2/25+201d/2+2\leq-9d^2/20$ for $d\geq10000$.
\end{proof}
\subsection{Extensions of roots with small Frattini subgroup}
\label{subsec:small-root-drift}

\begin{lemma}\label{lem:small-root-drift}
Let $Q$ be a root, with $d=d(Q)\geq64$ and $1\leq s=\dim\Phi(Q)\leq d/16$.
Let $Y$ be a finite $p$-group with $Y/\gamma_3^{(p)}(Y)\cong Q$ and depth $j=\log_p|\gamma_3^{(p)}(Y)|\leq d/8$.
If $Y$ has at least one allowable scalar line, then uniform sampling from its allowable scalar lines gives
\[
\Pr_{\mathrm{adm}}(a_Y(\xi)+\epsilon_Y(\xi)>0)\leq p^{-\frac{d^2}{4}},
\]
where $\xi$ is any nonzero representative of the sampled line.
\end{lemma}
\begin{proof}
Use $V,K_Y,Z_Y$ and the bound $\dim Z_Y\leq K_Y$ from the proof of Proposition~\ref{prop:annihilator-tail}.
Let $T=H^1(V)$ and $N_d=\binom d2$; inflation identifies $T$ with $H^1(Y)$.
In the elementary abelian cohomology algebra, take a basis with $\chi=x_1$ for any $0\ne\chi\in T$.
In odd characteristic the exterior part of the image of cup multiplication $\chi\cup-:H^2(V)\to H^3(V)$ has dimension $\binom{d-1}{2}$, and its mixed part has independent basis $x_1\operatorname{Bock}x_i$, $1\leq i\leq d$.
The rank of cup multiplication by $\chi$ is therefore $\binom{d-1}{2}+d=N_d+1$.
At $p=2$, multiplication by a nonzero linear polynomial is injective, with rank $D_d$.

For every $\lambda\in T$, consider $\operatorname{Bock}-\lambda\cup-:H^2(V)\to H^3(V)$.
For odd $p$, restrict the input to $\Lambda^2T$ and project the output onto $T\otimes\operatorname{Bock}T$.
The cup term disappears and
\[
\operatorname{Bock}(x_ix_j)
=x_j\operatorname{Bock}x_i-x_i\operatorname{Bock}x_j\quad(i<j)
\]
gives $N_d$ independent images, since each pair has its own mixed coordinates.
At $p=2$, write $\partial=\operatorname{Bock}$.
It satisfies $\partial x_i=x_i^2$, the derivation law and $\partial^2=0$.
If $\lambda=0$, it kills squares and sends $x_ix_j$ to $x_i^2x_j+x_ix_j^2$, again $N_d$ independent images.
For $\lambda\ne0$, change basis to $\lambda=x_1$ and write a quadratic polynomial uniquely as $f=cx_1^2+x_1\ell+q$, where $\ell,q$ involve only $x_2,\ldots,x_d$ and $q$ includes square terms.
Then
\[
(\partial+x_1)f=cx_1^3+x_1(\partial\ell+q)+\partial q.
\]
The kernel of $\partial+x_1$ consists exactly of the quadratic polynomials with $c=0,q=\partial\ell$, by comparing powers of $x_1$ and using $\partial^2=0$.
This kernel has dimension $d-1$, so the rank is $D_d-(d-1)=N_d+1$.

Cup products and Bocksteins commute with inflation.
The operators $\chi\cup-$ and $\operatorname{Bock}-\lambda\cup-$ therefore descend through $I_2:H^2(V)\to H^2(Y)$, whose image is $L(Y)$.
Degree-three inflation loses at most $K_Y$ dimensions, so each operator on $L(Y)$ has rank at least $R_Y=\max(0,N_d-K_Y)$.
On any affine coset $\xi_0+L(Y)$, the equation $\chi\cup\xi=0$ has, when soluble, proportion at most $p^{-R_Y}$ of solutions.
The same holds for $\operatorname{Bock}_Y(\xi)=\lambda\cup\xi$ at fixed $\lambda$.
Union over the $(p^d-1)/(p-1)$ lines for $\chi\ne0$ and over all $p^d$ values of $\lambda$ gives
\[
\Pr_{\xi_0+L(Y)}(a_Y(\xi)+\epsilon_Y(\xi)>0)
\leq\min\left\{1,\left(\frac{p^d-1}{p-1}+p^d\right)p^{-R_Y}\right\}.
\]
The events $a_Y(\xi)>0$ and $\epsilon_Y(\xi)>0$ are invariant under nonzero scaling, so the affine-coset argument in Proposition~\ref{prop:annihilator-tail} passes this estimate to uniform allowable lines whenever the set of allowable lines is nonempty.

The bounds $s\leq d/16$, $j\leq d/8$ and $h(Y)\geq0$ imply
\[
K_Y\leq3d^2/16+d/16,\quad R_Y\geq5d^2/16-9d/16.
\]
As $(p^d-1)/(p-1)+p^d\leq p^{d+1}$, the probability is at most $p^{-5d^2/16+25d/16+1}\leq p^{-d^2/4}$, using $d^2-25d-16\geq0$ for $d\geq64$.
\end{proof}

\subsection{Descendants of a fixed root}
\label{subsec:descendant-count}

Write $m=d+s+t$ and $\delta=m-3d/2$, and define
\[
U=\lceil2\delta_++10\rceil,\quad
w=\lceil(\delta-s+1000)_+\rceil,
\]
\[
E_p(d,s,t)=(2s+1)w(s+w)+50U(U+1)+t\log_p(d+3).
\]

\begin{proposition}\label{prop:descendant-count}
Suppose $d\geq10000$, $s,t\geq1$, and $|\delta|\leq d/1000$.
Let $b_0=D_d-s$.
Then $0\leq w<t$ and, for every $Q\in\mathcal R_p(d,s)$,
\[
D_p(Q,t)\leq\frac{p^{t(b_0+\nu(Q))+E_p(d,s,t)}}{(p-1)^t}.
\]
If $\nu(Q)=0$, then $D_p(Q,t)=0$.
\end{proposition}
\begin{proof}
Let $\nu=\nu(Q)$.
If $\nu=0$, the root $Q$ has no allowable scalar line, so $D_p(Q,t)=0$.
We may therefore assume that $\nu>0$.
At each vertex, attach one child for each allowable scalar line.
The central filtration described after Lemma~\ref{lem:allowable-extensions} reaches every group counted by $D_p(Q,t)$ at a depth-$t$ leaf; repeated abstract groups only enlarge the count.
At depth $i$ the branching number is exactly
\[
n_i=p^{b_0}\frac{p^{h_i}-1}{p-1}.
\]
A vertex with $h_i=0$ has no child.
Choosing children uniformly when possible gives the finite-tree identity
\[
\#\{\text{depth-}t\text{ leaves}\}
=\mathbb E\left[\mathbf1_{\mathrm{survival}}\prod_{i<t}n_i\right],
\]
because a surviving path has probability $\prod_{i<t}n_i^{-1}$.

Let $u_i=t-1-i$.
Since $s+i=d/2+\delta-1-u_i$ and $K_Y=d(s+i)+s-h_i$, every relevant vertex satisfies $K_Y<3d^2/5$, so the tails of Corollary~\ref{cor:annihilator-estimates} apply.
The elementary-abelian cup-rank calculation in the proof of Lemma~\ref{lem:small-root-drift}, which does not use the parameter restrictions in that lemma, gives $\rank(\chi\cup-)\geq\binom{d-1}{2}$ for $\chi\ne0$.
Inflation and a union over character lines therefore give
\[
\Pr(a_i>0)\leq\min\left\{1,\frac{p^d-1}{p-1}
p^{-\binom{d-1}{2}+K_Y}\right\}.
\]
Let $\alpha=d/2-1000$ and $z_i=\min(u_i,\alpha)$ for $0\leq i<t$.
Suppose first that $u_i\geq U$.
For $1\leq a_i<q_0=\lfloor d/100\rfloor$, the bound for $\Pr(a_i>0)$ gives
\[
\mathbb E[p^{z_i a_i}\mathbf1_{1\leq a_i<q_0}]
\leq p^{-\frac{99du_i}{100}+d\delta+2d+\delta-1}\leq p^{-d}.
\]
For the last inequality use $u_i\geq2\delta+10$ when $\delta\geq0$, and $u_i\geq10$ when $\delta<0$.
For $a_i\geq q_0$, the multidimensional tail gives
\[
\mathbb E[p^{z_i a_i}\mathbf1_{a_i\geq q_0}]
\leq\sum_{a=q_0}^dp^{(z_i-\alpha)a}\leq d+1.
\]
Including $a_i=0$, the conditional moment is at most $d+3$.
If $u_i<U$, the event $a_i<100$ contributes at most $p^{100u_i}$, and its complement contributes at most $p^{u_i d-9d^2/20}\leq1$.
Thus the moment is at most $2p^{100u_i}$.
The sum of the suffix weights $u_i<U$ is at most $U(U-1)/2$.
To justify multiplication of these conditional bounds, let $\Psi_i(Y)$ be the expected value of
\[
\mathbf1_{\mathrm{survival}}\prod_{i\leq k<t}p^{z_k a_k}
\]
for a path starting at a vertex $Y$ of depth $i$.
We have $\Psi_t(Y)=1$, and $\Psi_i(Y)=0$ if $Y$ has no child and $i<t$.
Otherwise,
\[
\Psi_i(Y)=\frac1{n_i}\sum_{\ell}
p^{z_i a_Y(\ell)}\Psi_{i+1}(E_\ell),
\]
where the sum is over the allowable scalar lines at $Y$, and $E_\ell$ is the chosen extension representing $\ell$.
Let
\[
c_i=\begin{cases}
d+3&\text{if }u_i\geq U\\
2p^{100u_i}&\text{if }u_i<U.
\end{cases}
\]
The bounds for $u_i\geq U$ and $u_i<U$ give a conditional expectation of $p^{z_i a_i}$ at most $c_i$ at every vertex of depth $i$.
Induction from depth $t$ therefore gives $\Psi_i(Y)\leq\prod_{i\leq k<t}c_k$.
In particular,
\[
\mathbb E\left[\mathbf1_{\mathrm{survival}}
\prod_{i<t}p^{z_i a_i}\right]
\leq(d+3)^tp^{50U(U+1)}.
\]
No independence between successive extensions is required.

Lemma~\ref{lem:central-cohomology} gives $h_{i+1}=h_i-1+a_i+\epsilon_i$, $0\leq\epsilon_i\leq1$.
Thus $h_{i+1}\leq h_i+a_i$, and a capacity $h_n\leq B+An$ gives
\[
\sum_{i<n}a_i=h_n-\nu+n-\sum_{i<n}\epsilon_i\leq B+(A+1)n.
\]
Take $B=(2s+1)s$, $A=2s$, which follows from Lemma~\ref{lem:capacity-inflation}.
Now $u_i-\alpha=\delta-s+999-i$ has positive part only for $i<w$, and that positive part is at most $w$.
Hence pathwise
\[
\sum_i(u_i-\alpha)_+a_i\leq(2s+1)w(s+w).
\]
Since $t=d/2+\delta-s$ and $d\geq10000$, the definitions also give $0\leq w<t$.
Finally $\sum_{i<t}h_i\leq t\nu+\sum_{i<t}u_i a_i$.
Insert this into the tree identity and use $n_i\leq p^{b_0+h_i}/(p-1)$, the clipped moments and the pathwise capacity bound to obtain the result.

\end{proof}

\subsection{Roots with bounded Frattini subgroup}
\label{subsec:bounded-frattini}

\begin{proposition}\label{prop:bounded-frattini}
Fix a prime $p$ and an integer $S\geq1$.
Let $N_{p,\leq S}(m)$ count the order-$p^m$ classes with $\dim\Phi(G/\gamma_3^{(p)}(G))\leq S$.
Then
\[
N_{p,\leq S}(m)\leq p^{\frac{2}{27}m^3-c_Sm^2\log_2m},\quad
c_S=\frac1{100\log_2(2S+1)}
\]
for sufficiently large $m$, with threshold allowed to depend on $p,S$.
In particular, $S=2$ gives saving $\ell_m m^2/100$, where $\ell_m=\log m/\log5$.
\end{proposition}
\begin{proof}
If $s=0$, then $\gamma_2^{(p)}(G)=\gamma_3^{(p)}(G)$; recursion and termination make this subgroup trivial, so $G$ is elementary abelian, one class at each order.
For $1\leq s\leq S$, let $B_s=s(s+3)/2$ and $A=2s$.
Lemma~\ref{lem:capacity-inflation} gives $h_j\leq B_s+Aj$, and branching is at most $p^{b_0+B_s+Aj}$, with $b_0=D_d-s$, since $(p-1)^{-1}\leq1$.

Call a step exceptional when $a_j+\epsilon_j>0$; every other step decreases $h$ by one.
If $d\geq\max(64,16S)$, then $s\leq d/16$.
At depths $j\leq d/8$, Lemma~\ref{lem:small-root-drift} bounds the conditional exceptional fraction by $p^{-d^2/4}$.
For $t\geq1$, if a surviving length-$t$ path has $k$ exceptional steps among its first $J\leq\min(t-1,\lfloor d/8\rfloor)$ steps, with the exceptional steps ending at depths $u_1<\cdots<u_k$, survival forces
\[
u_1\leq B_s,\quad u_{i+1}\leq B_s+(A+1)u_i,\quad
J<\frac{B_s((A+1)^{k+1}-1)}A.
\]
A path surviving to depth $t$ therefore requires at least
\[
k_J=\max\left\{0,
\left\lceil\frac{\log(1+AJ/B_s)}{\log(A+1)}\right\rceil-1\right\}
\]
exceptions.
For each specified subset of $k_J$ positions, successive conditioning gives a factor at most $p^{-d^2/4}$ at each position.
The tree identity in the proof of Proposition~\ref{prop:descendant-count}, the branching bound $p^{b_0+B_s+Aj}$ at depth $j$, and at most $2^J\leq p^J$ choices of exceptional positions therefore give
\[
D_p(Q,t)\leq p^{J+t(b_0+B_s)+st(t-1)-\frac{k_Jd^2}{4}}.
\]
At every $d$, including the finitely many values below $\max(64,16S)$, the capacity bound alone gives $D_p(Q,t)\leq p^{t(b_0+B_s)+st(t-1)}$.

There are at most $p^{sD_d}$ roots by Lemma~\ref{lem:root-generator-tail}.
At $t=0$ the empty path has count one; use this root bound directly and include the case $t=0$ in the range $t<m/100$ below.
The constants in the following $O$-terms depend only on $S$.
At fixed $d+s+t=m$, the rough exponent is $(m-d)d^2/2+O(m^2)$, uniformly for $1\leq s\leq S$.
For $d\leq m/2$ the leading term $(m-d)d^2/2$ is at most $m^3/16$.
For $t<m/100$, eventually $s+t<m/50$, so the leading term is at most $m^3/100$.
Both ranges have a fixed cubic gap.
Otherwise $d>m/2$ and $t\geq m/100$.
Eventually $d\geq\max(64,16S)$; take the largest permitted $J$, which satisfies $J\geq m/200$.
Then $k_J\geq\log_2m/\log_2(2S+1)-O(1)$ and $d^2/4>m^2/16$.
Since the unrestricted leading term is at most $2m^3/27$, the fixed-parameter logarithm is at most
\[
\frac{2m^3}{27}-\frac{m^2\log_2m}{16\log_2(2S+1)}+O(m^2).
\]
The factor $p^J$, the error terms, and the at most $S(m+1)$ parameter choices are absorbed by the growing quadratic saving.
Weakening $1/16$ to $1/100$ and including $s=0$ proves the claim.
\end{proof}

\section{Proof of the density theorem}\label{sec:density}

\begin{lemma}\label{lem:weighted-curvature}
Use the parameter range of Proposition~\ref{prop:descendant-count} and the quantities $U,w,E_p(d,s,t)$ defined in Subsection~\ref{subsec:descendant-count}, and set $z=s+t$.
For each fixed integer $k\geq0$,
\[
\mathcal F(d,z)+E_p(d,s,t)+kd(m-d)
\leq\lambda(m)+\frac{2k}{9}m^2+O(d\log d).
\]
The implied constant and a sufficiently large dimension threshold depend only on $k$, uniformly over primes and all parameters in the strip.
\end{lemma}
\begin{proof}
Let $y=\delta_++1001$.
When $w>0$, use $s+w\leq y$ and
\[
(2s+1)(y-s)=y^2/2+y/2+1/8-2(s-y/2+1/4)^2
\]
to obtain $(2s+1)w(s+w)\leq y^3/2+y^2$; when $w=0$ this term vanishes.
Since $U\leq2\delta_++1001$, first fix an absolute $K$ with
\[
(2s+1)w(s+w)+50U(U+1)
\leq\delta_+^3/2+K(\delta_+^2+\delta_++1).
\]
All subsequent dimension thresholds are chosen after fixing $K$.
The exact identities, with $m=3d/2+\delta$, are
\[
\begin{aligned}
\mathcal F(d,z)-\lambda(m)
&=5d\delta/6-(d/3+5/9)\delta^2-2\delta^3/27,\\
d(m-d)-2m^2/9&=d\delta/3-2\delta^2/9.
\end{aligned}
\]
If $\delta<0$, both differences are nonpositive in the strip: $d/3+5/9+2\delta/27>0$.
The added error apart from $t\log_p(d+3)$ is then at most $K$.
If $\delta\geq0$, the positive cubic error satisfies $\delta^3/2\leq d\delta^2/2000$.
Choose an absolute dimension threshold after $K$ so that its fixed linear and quadratic coefficients are absorbed by $d\delta/6$ and the remaining negative quadratic term.
The total difference, apart from $t\log_p(d+3)$, is then at most
\[
(1+k/3)d\delta-d\delta^2/4+K
\leq(1+k/3)^2d+K.
\]
For both signs of $\delta$, the total difference $\mathcal F(d,z)+E_p(d,s,t)+kd(m-d)-\lambda(m)-(2k/9)m^2$ is therefore at most $(1+k/3)^2d+K+t\log_p(d+3)$.
Finally $t<d$ gives $t\log_p(d+3)\leq d\log_2(d+3)$, uniformly in $p$.
\end{proof}

\subsection{Counting nonroots}
\label{subsec:nonroot-count}

\begin{proposition}\label{prop:nonroot-count}
For every fixed prime $p$, there is $m_0(p)$ such that
\[
N_p(m)\leq p^{\lambda(m)-\frac{m^2}{300}}
\quad(m\geq m_0(p)).
\]
\end{proposition}
\begin{proof}
Let $Q=G/\gamma_3^{(p)}(G)$.
We use the canonical parameters $d=d(G)$, $s=\dim\Phi(Q)$ and $t=\log_p|\gamma_3^{(p)}(G)|\geq1$.
Partition the nonroots of order $p^m$ into the three disjoint families
\[
s\leq2;\quad
s\geq3,\ |d/m-2/3|\geq1/10000;\quad
s\geq3,\ |d/m-2/3|<1/10000.
\]
By Proposition~\ref{prop:bounded-frattini}, the first contributes eventually at most $p^{\lambda(m)-m^2}$, since $\ell_m\to\infty$.
Proposition~\ref{prop:generator-localization} with fixed $\varepsilon=1/10000$ bounds the second by $p^{(2/27-c)m^3+O(m^{8/3})}\leq p^{\lambda(m)-m^2}$ eventually, for fixed $c>0$.

In the central strip $d>3m/5$ and
\[
\frac{|\delta|}{d}
=\frac32\frac{|d/m-2/3|}{d/m}<\frac1{1000}.
\]
Both Propositions~\ref{prop:weighted-roots} and~\ref{prop:descendant-count} apply for large $m$.
Let $z=s+t$ and discard $(p-1)^{-t}\leq1$ in the fiber estimate.
Roots with zero nucleus have no descendants.
Summing over the roots in $\mathcal R_p(d,s)$ with nonzero nucleus gives, for each triple $(d,s,t)$,
\[
\operatorname{count}(d,s,t)
\leq p^{\mathcal F(d,z)-\frac{d^2}{20}+E_p(d,s,t)}.
\]
Lemma~\ref{lem:weighted-curvature} with $k=0$ gives, after increasing the dimension threshold,
\[
\mathcal F(d,z)-\lambda(m)+E_p(d,s,t)\leq d^2/100.
\]

Each central-strip triple therefore contributes at most $p^{\lambda(m)-d^2/25}\leq p^{\lambda(m)-m^2/100}$, since $m<2d$.
There are at most $(m+1)^2$ triples $(d,s,t)$; the sum of their contributions is eventually at most $p^{\lambda(m)-m^2/200}$.
Combining the three families gives, for all sufficiently large $m$,
\[
N_p(m)
\leq2p^{\lambda(m)-m^2}+p^{\lambda(m)-\frac{m^2}{200}}
\leq p^{\lambda(m)-\frac{m^2}{300}}.\qedhere
\]
\end{proof}

\subsection{Completion of the proof}
\label{subsec:density-conclusion}

\begin{proof}[\rm\textbf{Proof of Theorem~\ref{thm:structure}}]
For sufficiently large $m$, Lemma~\ref{lem:root-lower-bound} and Proposition~\ref{prop:nonroot-count} give
\[
0\leq\frac{P_p(m)-R_p(m)}{R_p(m)}
=\frac{N_p(m)}{R_p(m)}\leq p^{-\frac{m^2}{300}}.
\]
Together with $R_p(m)\leq Q_p(m)\leq P_p(m)$, this proves the theorem.
\end{proof}

\section{Sharp enumeration and the generator distribution}\label{sec:sharp-enumeration}

Throughout this section, $p$ is fixed.
We first count roots with a given number of generators, and then sum over the integer values of that parameter.

\begin{proof}[\rm\textbf{Proof of Theorems~\ref{thm:sharp} and \ref{thm:generator}}]
We first prove the counting and distribution estimates for $R_p(m)$.
Proposition~\ref{prop:nonroot-count} will then give the estimates for $Q_p(m)$ and $P_p(m)$.

Use the root counts $B_p(d,s)$ and complete-parameter notation of Subsection~\ref{subsec:parameters}, so that $R_p(m)=\sum_{d=0}^m B_p(d,m-d)$, and write $D=D_d$.
For $d,s\geq1$ and $s\leq D$, Lemma~\ref{lem:root-orbit-identity} gives \eqref{eq:sharp-orbit-identity} with $0\leq\zeta\leq\eta_{p,d,s}$.
For any fixed $0<r\leq R$, the tensor estimate in Corollary~\ref{cor:parameter-errors} gives constants $c>0,d_0$, depending only on $r,R$, such that
\[
 0\leq\Delta_{p,d,s}\leq p^{-cd^2}\quad
 (d\geq d_0,\ rd\leq s\leq Rd).
\]

Use $a_{p,n}$ and $T_n\coloneqq T_{p,n}$ from \eqref{eq:product-bounds}.
Equation~\eqref{eq:sharp-orbit-identity}, with $E=sD-d^2-s^2$, yields
\[
 \frac{B_p(d,s)}{g_p^{-2}p^E}=(\sigma_{p,d,s}+\zeta)T_dT_s.
\]
Thus the two general linear group orders account for the factor $g_p^{-2}$.
Since
\[
 1-\sigma_{p,d,s}\leq\frac{p^{s-D}}{p-1},\quad
 1-T_dT_s\leq\frac{p^{-d}+p^{-s}}{p-1},
\]
we obtain the finite relative-error bound
\begin{equation}\label{eq:sharp-layer-error}
 \left|\frac{B_p(d,s)}{g_p^{-2}p^{sD-d^2-s^2}}-1\right|
 \leq\Delta_{p,d,s}+(p-2)p^{-ds}
      +\frac{p^{s-D}+p^{-d}+p^{-s}}{p-1}.
\end{equation}

Set
\[
 E_m(d)=\frac{(m-d)d(d+1)}2-d^2-(m-d)^2,\quad
 \mathcal J_m=\{d\in\mathbf Z:3m/5\leq d\leq3m/4\}.
\]
For $d\in\mathcal J_m$ and $s=m-d$,
\[
 \frac m4\leq s\leq\frac{2m}5,\quad d\geq\frac{3m}5,\quad
 \frac13\leq\frac sd\leq\frac23,\quad ds\geq\frac{3m^2}{20},\quad
 D-s\geq\frac{9m^2}{50}-\frac m{10}.
\]
The right side of \eqref{eq:sharp-layer-error} is therefore at most
\[
 p^{-\frac{9cm^2}{25}}+(p-2)p^{-\frac{3m^2}{20}}
 +\frac{p^{-\frac{9m^2}{50}+\frac{m}{10}}+p^{-\frac{3m}{5}}+p^{-\frac{m}{4}}}{p-1}.
\]
In particular,
\begin{equation}\label{eq:sharp-central-upper}
 B_p(d,m-d)=g_p^{-2}p^{E_m(d)}(1+O(p^{-\frac{m}{4}})),\quad
 B_p(d,m-d)\leq2g_p^{-2}p^{E_m(d)}
\end{equation}
uniformly on $\mathcal J_m$, once $m$ is sufficiently large.

We next determine the integers at which $E_m(d)$ is largest.
Expanding gives
\[
 E_m(d)=\frac{-d^3+(m-5)d^2+5md-2m^2}{2},
\]
and, writing $\Delta_m(d)=E_m(d+1)-E_m(d)$,
\[
 \Delta_m(d)=-\frac32d^2+\left(m-\frac{13}2\right)d+3m-3,
 \quad \Delta_m(d+1)-\Delta_m(d)=m-3d-8.
\]
The differences strictly decrease for $d\geq3m/5$.
The table gives the set $\mathcal T_m$ of dominant generator numbers, the maximum $e_m$ on $\mathcal J_m$, and the outer gaps $\ell_m=\Delta_m(d_--1)$, $r_m=-\Delta_m(d_+)$, where $d_-\coloneqq\min\mathcal T_m$ and $d_+\coloneqq\max\mathcal T_m$:
\begin{equation}\label{eq:sharp-saddle-table}
 \begin{array}{c|c|c|c|c}
 m&\mathcal T_m&e_m-\lambda(m)&\ell_m&r_m\\ \hline
 3a&\{2a-1\}&m/3-2&2a+4&a-2\\
 3a+1&\{2a\}&m/3-26/27&a+4&2a\\
 3a+2&\{2a,2a+1\}&2m/9-7/27&3a+6&3a+3
 \end{array}
\end{equation}
For $m\geq60$, comparison with $3m/5$ and $3m/4$ places these layers and their immediate neighbors in $\mathcal J_m$.
The unique maxima are $2a-1$, $2a$, and $2a+1$, respectively; in the last case $\Delta_m(2a)=3$ and $E_m(2a)=e_m-3$.
Set $h_m=\min\{\ell_m,r_m\}$, equal respectively to $a-2,a+4,m+1$ (using $2a\geq a+4$ in the middle case).

Let
\[
 S_m=\sum_{d\in\mathcal T_m}p^{E_m(d)},\quad
 R_m=\sum_{d\in\mathcal J_m\setminus\mathcal T_m}p^{E_m(d)}.
\]
Then $S_m=p^{e_m}$ in the first two cases and $S_m=(1+p^{-3})p^{e_m}$ in the third, so $M_p(m)=g_p^{-2}S_m$.
Monotonicity of the differences gives, for $j\geq1$ whenever the indicated point lies in $\mathcal J_m$,
\[
 \begin{aligned}
 E_m(d_--j)&\leq E_m(d_-)-j\ell_m\leq e_m-jh_m,\\
 E_m(d_++j)&\leq e_m-jr_m\leq e_m-jh_m.
 \end{aligned}
\]
Summing the two geometric tails, and using $h_m\geq m/3-2\geq m/4$ for $m\geq60$, gives
\begin{equation}\label{eq:sharp-geometric-tail}
 \frac{R_m}{S_m}\leq\frac{2p^{-h_m}}{1-p^{-h_m}}
 \leq4p^2p^{-\frac{m}{3}}.
\end{equation}
The normalized weights on $\mathcal T_m$ are exactly those in \eqref{eq:generator-law}; in particular, the two weights in the third case have ratio $p^3:1$.

Apply Lemma~\ref{lem:root-generator-tail}, including its boundary cases $d=0$, $s=0$, and $s>D$, with $\rho=1/15$.
Since the table gives $M_p(m)\geq p^{\lambda(m)}$ for $m\geq60$, the tail outside $\mathcal J_m$, divided by $M_p(m)$, is at most
\[
 p^{-\frac{m^3}{2700}+\frac{4}{9}m^2}\leq p^{-\frac{m^3}{5400}}\leq p^{-\frac{m}{3}}
\]
for $m\geq2400$, subject also to the lemma's sufficiently-large condition.
Hence
\begin{equation}\label{eq:sharp-distant-tail}
 \sum_{d\notin\mathcal J_m}B_p(d,m-d)\leq M_p(m)p^{-\frac{m}{3}}.
\end{equation}

On the principal layers in \eqref{eq:sharp-saddle-table}, let $s=m-d$.
Then $d,s\geq m/3$, and, for $m\geq60$, $d^2\geq m^2/4$, $ds\geq m^2/6$, $D-s\geq m^2/16$.
Thus \eqref{eq:sharp-layer-error}, used again at these layers, gives
\[
 \left|\frac{B_p(d,m-d)}{g_p^{-2}p^{E_m(d)}}-1\right|
 \leq p^{-\frac{cm^2}{4}}+(p-2)p^{-\frac{m^2}{6}}
       +\frac{p^{-\frac{m^2}{16}}+2p^{-\frac{m}{3}}}{p-1}
 =O(p^{-\frac{m}{3}}).
\]
Define $b_{m,d}=B_p(d,m-d)$ for $0\leq d\leq m$ and zero otherwise, and
\[
 w_{m,d}=g_p^{-2}p^{E_m(d)}\mathbf1_{\{d\in\mathcal T_m\}}.
\]
The principal-layer estimate, \eqref{eq:sharp-central-upper}, \eqref{eq:sharp-geometric-tail}, and \eqref{eq:sharp-distant-tail} together show that
\begin{equation}\label{eq:sharp-l1}
 L_m\coloneqq\sum_{d\geq0}|b_{m,d}-w_{m,d}|
 \leq K(p)M_p(m)p^{-\frac{m}{3}}.
\end{equation}
In particular, the central nonprincipal layers contribute at most $8p^2M_p(m)p^{-m/3}$; their contribution is therefore within the required error.

Since $\sum b_{m,d}=R_p(m)$, $\sum w_{m,d}=M_p(m)$, and $w_{m,d}/M_p(m)=\pi_{p,m}(d)$, the single estimate \eqref{eq:sharp-l1} proves $R_p(m)=M_p(m)(1+O(p^{-m/3}))$ and the generator-distribution estimate for roots.
Indeed,
\[
 |R_p(m)-M_p(m)|\leq L_m,
\]
and, writing $A=R_p(m)$ and $M=M_p(m)$,
\[
 \sum_d\left|\frac{b_{m,d}}A-\frac{w_{m,d}}M\right|
 \leq\left|1-\frac AM\right|+\frac{L_m}M\leq\frac{2L_m}M.
\]
Thus the root total variation error is at most $L_m/M$.

Finally Proposition~\ref{prop:nonroot-count} gives, for each fixed $p$,
\[
P_p(m)-R_p(m)\leq p^{\lambda(m)-\frac{m^2}{300}}.
\]
Every root has $G'\leq\Phi(G)\leq Z(G)$, so $R_p\leq Q_p\leq P_p$.
For $Y\in\{P_p,Q_p\}$,
\[
 0\leq\frac{Y(m)-R_p(m)}{M_p(m)}\leq p^{-\frac{m^2}{300}},
\]
which proves $Y(m)=M_p(m)(1+O(p^{-m/3}))$.
The uniform measure on the isomorphism classes counted by $Y(m)$, conditional on the chosen group being a root, is uniform on the root isomorphism classes of order $p^m$.
Consequently
\[
 d_{\mathrm{TV}}(\mu_{Y,p,m},\mu_{R,p,m})
 \leq\frac{Y(m)-R_p(m)}{Y(m)}
 \leq2p^{-\frac{m^2}{300}},
\]
using $R_p(m)\geq M_p(m)/2$ for large $m$.
The triangle inequality, combined with the estimate for $\mu_{R,p,m}$, proves the generator-distribution estimate for $Y\in\{P_p,Q_p\}$.
\end{proof}

\subsection{A logarithmic consequence}\label{subsec:logarithmic-count}

\begin{corollary}\label{cor:logarithmic-count}
For every fixed prime $p$ and $X\in\{P_p,Q_p,R_p\}$,
\[
\log_p X(m)=\lambda(m)+O(m).
\]
Moreover,
\[
\limsup_{m\to\infty}\frac{\log_pX(m)-\lambda(m)}m=\frac13,
\quad
\liminf_{m\to\infty}\frac{\log_pX(m)-\lambda(m)}m=\frac29,
\]
and
\[
\log_p R_p(m)=\log_p Q_p(m)+O(p^{-\frac{m^2}{300}})
            =\log_p P_p(m)+O(p^{-\frac{m^2}{300}}).
\]
\end{corollary}
\begin{proof}
By Theorems~\ref{thm:sharp} and \ref{thm:structure}, the formula \eqref{eq:sharp-main-term} gives
\[
\frac{\log_p X(m)-\lambda(m)}m=
\begin{cases}
1/3+O(m^{-1})&\text{if }m\equiv0,1\pmod3\\
2/9+O(m^{-1})&\text{if }m\equiv2\pmod3.
\end{cases}
\]
The residue-class expansion proves the logarithmic asymptotic and the stated limsup and liminf.
For $R_p(m)\leq x\leq y\leq P_p(m)$, Theorem~\ref{thm:structure} gives
\[
0\leq\log_p y-\log_p x
\leq\log_p(1+p^{-\frac{m^2}{300}})
\leq\frac{p^{-\frac{m^2}{300}}}{\ln p}.
\]
This proves the final assertion.
\end{proof}

\section{Automorphism moments and finite-group enumeration}\label{sec:cumulative-enumeration}

We now prove Theorem~\ref{thm:main}.
The first step is to estimate fixed moments of automorphism-group orders for $2$-groups.
The automorphism-moment estimates lead to bounds for the number of groups of order $2^a q$ with a fixed odd cofactor $q>1$.
We combine these fixed-cofactor bounds with estimates uniform over all group orders to sum over a dyadic interval.

All logarithms in this section have base two unless another base is specified.
The moment order $k$ and the odd cofactor $q$ are fixed before the order exponent tends to infinity.
Let $\Omega(q)$ denote the number of prime factors of $q$, counted with multiplicity.

Specialize the complete parameter space of Subsection~\ref{subsec:parameters} to $p=2$, writing $Y(d,s)=\mathscr P_{2,d,s}$ and $Y(d,s)_{\mathrm{adm}}=\mathscr P_{2,d,s}^*$.
For a root $Q\in\mathcal R_2(d,s)$, let $V=Q/\Phi(Q)$, $W=\Phi(Q)$, and $S_Q=\operatorname{Stab}_{\GL(V)\times\GL(W)}(\beta,\pi)$.
Lemma~\ref{lem:parameter-counts} gives $|Y(d,s)|=2^{sD_d}$, and Proposition~\ref{prop:root-parameters} identifies the orbits in $Y(d,s)_{\mathrm{adm}}$ with the isomorphism classes in $\mathcal R_2(d,s)$.
The exact sequence in Lemma~\ref{lem:root-automorphisms} gives
\begin{equation}\label{eq:root-automorphism-order}
 |\Aut(Q)|=2^{ds}|S_Q|.
\end{equation}
The projection $S_Q\to\GL(V)$ is injective: an element fixing $V$ fixes every commutator and power value, and these jointly span $W$.
When $s=0$, these statements mean that $Y(d,0)$ is a singleton, $Q\cong(\F_2^d,+)$, and $S_Q=\Aut(Q)=\GL(d,2)$; for $s>D_d$ there are no admissible parameters.
A uniform complete parameter has uniform commutator marginal, since each quadratic-refinement fiber has $2^{ds}$ elements.
We may therefore bound a sum of nonnegative quantities over admissible parameters by the corresponding sum over the complete space.

\subsection{Moments of root stabilizers}\label{subsec:stabilizer-moments}

\begin{lemma}\label{lem:root-stabilizer-moments}
 Fix an integer $h\geq1$ and $0<\varepsilon\leq b$.
 For $X=\Hom(\Lambda^2\F_2^d,\F_2^s)$ and $\Gamma=\GL(d,2)\times\GL(s,2)$, uniformly for $\varepsilon d\leq s\leq bd$,
 \[
\mathbf E_{x\in X}|\Gamma_x|^h=1+O(2^{-c_hd^2})
\]
 for some $c_h>0$.
 The same assertion holds for the complete parameter space $Y(d,s)$.
 The constants and dimension threshold may depend on $h,\varepsilon,b$.
\end{lemma}
\begin{proof}
First use $X_0=\Lambda^2V\otimes W$, where $V=\F_2^d$ and $W=\F_2^s$.
Set $M=1+b^2$ and apply Lemma~\ref{lem:tensor-rank} with $a=\varepsilon$, $n=d$, $q=2$, and $T=hM+1$.
With its $K,B_K$ and $\kappa_0=\min\{\varepsilon/8,1/16\}$, Lemma~\ref{lem:fixed-space-moments} applies with kernel $Z=\{(I,I)\}$, $L=Td^2$ and $H=\kappa_0d^2$.
Since $|\Gamma|\leq2^{Md^2}$ and $|B_K|\leq2^{2K(d+s)}$, it gives the explicit bound
\[
1\leq\mathbf E_{x\in X_0}|\Gamma_x|^h
 \leq1+2^{-d^2}+2^{2hK(1+b)d-\kappa_0d^2}.
\]
For fixed $h,\varepsilon,b$, this proves the assertion, for example with $c_h=\kappa_0/2$ after increasing the common dimension threshold.

Now $X\cong\Lambda^2(V^*)\otimes W$.
Reindexing all source transformations by the automorphism $P\mapsto P^{-T}$ of $\GL(d,2)$ identifies its fixed-space sum with that for $X_0$.
Finally $\Gamma_{(\beta,\pi)}\leq\Gamma_\beta$, and the $2^{ds}$ quadratic refinements of each $\beta$ give
\[
 1\leq\mathbf E_{y\in Y(d,s)}|\Gamma_y|^h
 \leq\mathbf E_{\beta\in X}|\Gamma_\beta|^h.\qedhere
\]
\end{proof}

\begin{lemma}\label{lem:common-form-stabilizers}
 Choose $s$ forms independently and uniformly from $\Hom(\Lambda^2\F_2^d,\F_2)$, and let $Z_s$ be the order of their common stabilizer in $\GL(d,2)$.
 For fixed $h\geq1$, if $s\geq12h$ and $d$ is sufficiently large, then
 \[
\mathbf E Z_s^h\leq2.
\]
\end{lemma}
\begin{proof}
Write $\rho(g)=\rank(\Lambda^2g-I)$ and $r=\rank(g-I)$.
The fixed-space codimension on alternating forms is also $\rho(g)$.
If $1\leq r\leq d/2$, let $U=\ker(g-I)$ and $B=\im(g-I)$.
For $u\in U$,
\[
 (\Lambda^2g-I)(u\wedge v)=u\wedge(g-I)v,
\]
so its image contains $U\wedge B$.
If $j=\dim(U\cap B)$, then
\[
\rho(g)\geq\dim(U\wedge B)
 =(d-r)r-\frac{j(j+1)}2
 \geq r(d-r)-\frac{r(r+1)}2\geq r(d/4-1/2).
\]
For the dimension identity, write $J=U\cap B$, $U=J\oplus U_0$, and $B=J\oplus B_0$; the four summands $\Lambda^2J$, $J\wedge B_0$, $U_0\wedge J$, and $U_0\wedge B_0$ form a direct sum.
There are at most $2^{2dr}$ matrices of rank $r$ of the form $g-I$, since each factors as a $d\times r$ matrix times an $r\times d$ matrix.

For $r>d/2$, let $c_{\alpha,t}$ be the Jordan column lengths over an algebraic closure.
They sum to $d$ and are at most $d/2$: at $1$ the geometric multiplicity is $d-r<d/2$; other eigenvalues have Frobenius orbits of length at least two with identical blocks.
For $V_i=J_{\ell_i}(\alpha_i)$, invariant pairings $V_i\times V_j$ are homomorphisms to the dual block $J_{\ell_j}(\alpha_j^{-1})$, of dimension $h_{ij}=\mathbf1_{\alpha_i\alpha_j=1}\min(\ell_i,\ell_j)$.
The same block calculation and Cauchy--Schwarz give
\[
\begin{aligned}
 C\coloneqq\dim C_{\operatorname{End}(V)}(g)
   &=\sum_{\alpha,t}c_{\alpha,t}^2\leq d^2/2,\\
 H\coloneqq\sum_{i,j}h_{ij}
   &=\sum_{\alpha,t}c_{\alpha,t}c_{\alpha^{-1},t}\leq C.
 \end{aligned}
\]
Here $H$ is the full invariant bilinear-form dimension.
Alternating forms choose each off-diagonal pair only once, with diagonal freedom at most $h_{ii}\leq\ell_i$, whence
\[
 \dim\operatorname{Fix}_{\mathrm{Alt}}(g)
 \leq\sum_{i<j}h_{ij}+\sum_i h_{ii}
 =\frac{H+\sum_i h_{ii}}2\leq\frac{C+d}{2}.
\]
Scalar extension preserves ranks, so
\[
\rho(g)\geq\binom d2-(d^2/4+d/2)=d^2/4-d.
\]
Combining the contribution from $g=I$ with the bounds for the two ranges $1\leq r\leq d/2$ and $r>d/2$ gives, for sufficiently large $d$,
\begin{equation}\label{eq:exterior-rank-sum}
 \sum_{g\in\GL(d,2)}2^{-12\rho(g)}
 \leq1+\sum_{1\leq r\leq d/2}2^{(-d+6)r}+2^{-2d^2+12d}
 \leq1+2^{-\frac{d}{2}}.
\end{equation}
The common fixed-space codimension is at least the average of the individual codimensions.
Independence of the $s$ sampled forms, followed by \eqref{eq:exterior-rank-sum}, therefore yields
\[
 \mathbf E Z_s^h
 =\sum_{g_1,\ldots,g_h}
 2^{-s\operatorname{codim}(\bigcap_i\operatorname{Fix}(g_i))}
 \leq\prod_{i=1}^h\left(\sum_g2^{-\frac{s}{h}\rho(g)}\right)
 \leq(1+2^{-\frac{d}{2}})^h\leq2.\qedhere
\]
\end{proof}

\begin{lemma}\label{lem:stabilizer-weight-sum}
 Let $B(d,s)=sD_d-d^2-s^2$ and $U_l(d,s)=\sum_{Q\in\mathcal R_2(d,s)}|S_Q|^l$.
 For each fixed integer $l\geq1$, there are integers $S_l,d_l$ such that
 \[
U_l(d,s)\leq2^{B(d,s)+d^2}
 \quad(d\geq d_l,\ S_l\leq s\leq0.51d).
\]
 One may take $S_l=12(l+1)$.
\end{lemma}
\begin{proof}
Set $h=l+1$.
Orbit--stabilizer on the admissible space gives
\begin{equation}\label{eq:weighted-orbit-identity}
 U_l(d,s)=\frac1{|\Gamma|}
 \sum_{y\in Y(d,s)_{\mathrm{adm}}}|S_y|^{l+1}
 \leq\frac1{|\Gamma|}\sum_{y\in Y(d,s)}|S_y|^{l+1}.
\end{equation}
Project $S_y$ to $\GL(s,2)$.
Its kernel is contained in the common source stabilizer of the $s$ coordinate alternating forms of $\beta$; it must also preserve $\pi$.
Thus $|S_y|\leq2^{s^2}Z_s$.
Under uniform sampling from $Y(d,s)$ these coordinate forms are independent and uniform; independence here is probabilistic and does not assert linear independence.
By Lemma~\ref{lem:common-form-stabilizers} and $|\GL(n,2)|\geq2^{n^2}/4$,
\begin{equation}\label{eq:form-moment-bound}
 U_l(d,s)\leq2^{B(d,s)+(l+1)s^2+5}
 \quad(s\geq12(l+1))
\end{equation}
for a dimension threshold depending only on $l$.
The constant $5$ accounts for the two general linear group denominators and the factor $2$ in the form moment.

Take $\varepsilon_l=1/(2(l+1))$.
If $s<\varepsilon_l d$, then $(l+1)s^2\leq d^2/2$, so \eqref{eq:form-moment-bound} proves the result for $d\geq4$.
If $\varepsilon_l d\leq s\leq0.51d$, Lemma~\ref{lem:root-stabilizer-moments} makes the complete-parameter $h$th moment at most $2$, uniformly after increasing the threshold.
Equation \eqref{eq:weighted-orbit-identity} then gives $U_l(d,s)\leq2^{B(d,s)+5}$.
Taking $d_l$ to be the maximum of $4$ and these two thresholds proves the assertion uniformly over the entire stated range.
\end{proof}
\subsection{Weighted counts of descendants}\label{subsec:weighted-descendants}

Let $\nu(Q)$ be the nucleus dimension and let $R=R_\beta$ be the maximum rank in the Jacobi dual, as defined in Subsection~\ref{subsec:nucleus}.
We use the nucleus bounds from Subsection~\ref{subsec:nucleus} together with the root stratum estimates \eqref{eq:root-rank-bounds} and the individual weighted margins in Proposition~\ref{prop:weighted-roots}, before summing over $R$.
The parameter range, $z=s+t$, and $\mathcal F=zD_d-d^2-z^2$ are as in \eqref{eq:root-weight-strip}.

\begin{lemma}\label{lem:weighted-roots-with-stabilizers}
 Fix $k\geq1$.
 There is $S(k)$ such that, in \eqref{eq:root-weight-strip}, for $s\geq S(k)$ and sufficiently large $d$,
 \[
\sum_{\substack{Q\in\mathcal R_2(d,s)\\\nu(Q)>0}}
 |S_Q|^k2^{t(D_d-s)+t\nu(Q)}
 \leq2^{zD_d-d^2-z^2-\frac{d^2}{25}}.
\]
\end{lemma}
\begin{proof}
Write $n_R\leq2^{A_R}$ for the appropriate bound in \eqref{eq:root-rank-bounds}, and $\nu(Q)\leq u_R$ for the nucleus bound on the rank-$R$ stratum.
Take $h=1000$.
Empty strata contribute zero.
Since $s\leq z\leq0.501d<0.51d$, Lemma~\ref{lem:stabilizer-weight-sum}, with $l=kh$, applies for $s\geq S_{kh}$ and sufficiently large $d$.
H\"older's inequality gives
\[
\sum_{Q\text{ in stratum }R}|S_Q|^k
 \leq n_R^{1-\frac{1}{h}}U_{kh}(d,s)^{\frac{1}{h}}
 \leq2^{A_R+\frac{B(d,s)+d^2-A_R}{h}}.
\]
For each of the three bounds in \eqref{eq:root-rank-bounds}, the additional exponent is controlled by
\begin{equation}\label{eq:holder-cost}
\begin{aligned}
B(d,s)-A_R&\leq\Delta_R\leq Rd^2/2,\\
\frac{B(d,s)+d^2-A_R}{h}
&\leq\frac{\Delta_R+d^2}{1000}\leq\frac{R/2+1}{1000}d^2.
\end{aligned}
\end{equation}
Combining \eqref{eq:holder-cost} with the $d^2/16$ saving for $R\leq4$ in Proposition~\ref{prop:weighted-roots} and \eqref{eq:root-rank-margin} otherwise, the weighted stratum's logarithm minus $\mathcal F$ is at most
\[
 \begin{cases}
 -(1/16-3/1000)d^2=-0.0595d^2&\text{if }R\leq4\\
 (1.2611-0.2695R)d^2\leq-0.0864d^2&\text{if }R\geq5.
 \end{cases}
\]
There are at most $d$ strata.
Since $\log_2d\leq d^2/80$, summing the weighted contributions over the rank strata still leaves a saving greater than $d^2/25$.
Take $S(k)=\max\{3,S_{1000k}\}=S_{1000k}$.
\end{proof}

\subsection{Automorphism moments for all \texorpdfstring{$2$}{2}-groups}\label{subsec:automorphism-moments}

\begin{proposition}\label{prop:automorphism-moments}
 For each fixed nonnegative integer $k$, let $M_k(m)=\sum_{|P|=2^m}|\Aut(P)|^k$, where the sum is over abstract isomorphism classes.
 Then
 \begin{equation}\label{eq:automorphism-moments}
 \log_2M_k(m)=\lambda(m)+\frac{2k}{9}m^2+O(m).
 \end{equation}
\end{proposition}
\begin{proof}
For $k=0$ this is Corollary~\ref{cor:logarithmic-count}; fix $k\geq1$.
Use the central generator interval
\begin{equation}\label{eq:central-generator-range}
 |d/m-2/3|<1/10000.
\end{equation}
Here $d>3m/5$, $m<2d$, and $|m-3d/2|<d/1000$.
Proposition~\ref{prop:generator-localization} bounds the number of groups of order $2^m$ with generator number outside \eqref{eq:central-generator-range} by
\begin{equation}\label{eq:all-group-generator-tail}
 2^{(\frac{2}{27}-c_\varepsilon)m^3+O(m^{\frac{8}{3}})},\quad
 c_\varepsilon=\min\{7/3375,7\varepsilon^2/15\}>0,
 \quad\varepsilon=1/10000.
\end{equation}
The trivial bound $|\Aut(P)|^k\leq2^{km^2}$ leaves that contribution negligible.

For roots in \eqref{eq:central-generator-range}, Lemma~\ref{lem:root-stabilizer-moments} and \eqref{eq:weighted-orbit-identity}, followed by \eqref{eq:root-automorphism-order}, give
\[
 \sum_{Q\in\mathcal R_2(d,s)}|\Aut(Q)|^k
 \leq2^{E_m(d)+kds+O(1)},\quad s=m-d.
\]
Writing $d=2m/3+x$, the exact exponent difference is
\[
\begin{aligned}
 E_m(d)+kd(m-d)-\lambda(m)-\frac{2k}{9}m^2
 ={}&-\frac m2x^2-\left(\frac56+\frac k3\right)mx\\
    &-\frac12x^3-\left(\frac52+k\right)x^2.
 \end{aligned}
\]
For $|x|\leq m/10000$, its right side is at most
\[
-\frac m4x^2+\left(\frac56+\frac k3\right)m|x|
 \leq\left(\frac56+\frac k3\right)^2m.
\]
The at most $m+1$ layers cost $O(\log m)$, so the root contribution has the required upper bound.
Conversely, with $d=\lfloor2m/3\rfloor$, $s=m-d$, Lemma~\ref{lem:root-lower-bound} gives
\[
|\mathcal R_2(d,s)|\geq2^{s(D_d-s)-d^2}\geq2^{\lambda(m)}.
\]
Each root has at least $2^{ds}$ automorphisms and $ds=2m^2/9+O(m)$, proving
\[
M_k(m)\geq2^{\lambda(m)+\frac{2k}{9}m^2-O(m)}.
\]

We now bound the weighted descendants at every depth.
For arbitrary $P$ let $Q=P/\gamma_3^{(2)}(P)$ and $d=d(P)$.
Once the action on $P/\Phi(P)$ is fixed, each of $d$ generators has at most $2^{m-d}$ possible images.
Hence the kernel of the action of $\Aut(P)$ on $P/\Phi(P)$ has size at most $2^{d(m-d)}$.
The image of this action embeds in $S_Q$, since the action on $\Phi(Q)$ is determined by the generating commutator and power values.
Therefore
\begin{equation}\label{eq:descendant-automorphisms}
 |\Aut(P)|^k\leq2^{kd(m-d)}|S_Q|^k.
\end{equation}

For nonroots in the central range with $s\geq S(k)$ and $t\geq1$, the hypotheses of Proposition~\ref{prop:descendant-count} and Lemma~\ref{lem:weighted-curvature} hold once $d$ is sufficiently large.
Use their notation $E_2(d,s,t)$ and $\mathcal F$; roots with $\nu(Q)=0$ have no positive-depth descendants.
The arbitrary-depth fiber bound, Lemma~\ref{lem:weighted-roots-with-stabilizers}, and \eqref{eq:descendant-automorphisms} bound the contribution for fixed $(d,s,t)$ by
\begin{equation}\label{eq:weighted-descendant-count}
 2^{\mathcal F-\frac{d^2}{25}+E_2(d,s,t)+kd(m-d)}.
\end{equation}
Lemma~\ref{lem:weighted-curvature} absorbs the fiber error and weight into $O(d\log d)$, so for fixed $k$ and sufficiently large $d$,
\begin{equation}\label{eq:weighted-descendant-curvature}
 \mathcal F+E_2(d,s,t)+kd(m-d)
 \leq\lambda(m)+\frac{2k}{9}m^2+\frac{d^2}{100}.
\end{equation}
Equations \eqref{eq:weighted-descendant-count}--\eqref{eq:weighted-descendant-curvature} therefore bound this nonroot contribution by $2^{\lambda(m)+(2k/9)m^2-3d^2/100}$.
For $s<S(k)$, Proposition~\ref{prop:bounded-frattini} with $p=2$ and $S=S(k)$ applies at every depth: multiplication by $2^{km^2}$ is absorbed by its $c_Sm^2\log_2m$ saving, since $k$ is fixed.
Outside \eqref{eq:central-generator-range}, the same weight leaves the cubic saving in \eqref{eq:all-group-generator-tail}.
Both contributions therefore have a fixed quadratic saving.
In the central range, $3d^2/100>27m^2/2500$, while summing over at most $(m+1)^2$ triples costs $O(\log m)$.
After increasing the threshold,
\[
\sum_{\substack{|P|=2^m\\\gamma_3^{(2)}(P)\ne1}}|\Aut(P)|^k
 \leq2^{\lambda(m)+\frac{2k}{9}m^2-\frac{m^2}{1000}}.
\]
Choose a common threshold for the preceding estimates, with $k$ fixed.
Together with the root bounds, this proves \eqref{eq:automorphism-moments}, with a threshold and an $O(m)$ constant depending only on $k$.
\end{proof}

\subsection{Odd-order outer actions}\label{subsec:odd-actions}

We use the completely reducible linear-group bound \cite[Theorem~1.2]{Maroti2004}: if $B\leq\GL(d,Q)$ is completely reducible and $Q$ is a prime power, then $k(B)\leq Q^{5d}$.
The weaker $Q^{10d}$ bound suffices below.
This is a direct source for the needed estimate, an earlier form of which is given in \cite[Corollary~5]{LP1997}.

\begin{lemma}\label{lem:odd-outer-classes}
 If $P$ is a finite $d$-generated $2$-group, then
 \[
k_{2'}(\Out(P))\leq2^{10d},
\]
 where $k_{2'}$ counts conjugacy classes of odd-order elements, including the identity.
\end{lemma}
\begin{proof}
For $P=1$ both sides are one.
Otherwise the Frattini automorphism theorem \cite[Lemma~3.13]{BNV2007} says that the kernel of $\Aut(P)\to\GL(d,2)$ is a $2$-group.
The kernel contains $\Inn(P)$, so $\Out(P)$ has a normal $2$-subgroup with quotient isomorphic to a subgroup $B\leq\GL(d,2)$.

If $K\trianglelefteq A$ is a $2$-group, the quotient map induces a bijection
\[
\{\text{odd-order conjugacy classes in }A\}
 \longleftrightarrow
 \{\text{odd-order conjugacy classes in }A/K\}.
\]
Indeed, the inverse image $E$ of a cyclic subgroup $\langle\bar x\rangle$ of odd order has normal Hall subgroup $K$.
Schur--Zassenhaus \cite[Corollary~7.17 and the following discussion]{BNV2007} supplies a complement mapping isomorphically onto $\langle\bar x\rangle$, hence an odd-order lift of $\bar x$.
Two such lifts $x,y$ have the same order as $\bar x$, since their cyclic subgroups meet $K$ trivially.
These cyclic subgroups are complements in $E$ and are conjugate under $K$.
Since both elements map to $\bar x$, this conjugation also takes $x$ to $y$.
If the quotient elements are merely conjugate, first lift a conjugating element.
This proves the bijection.

Choose a $B$-invariant composition series $0=V_0<V_1<\cdots<V_t=V$.
The kernel $K_1$ of the action of $B$ on $\bigoplus_i V_i/V_{i-1}$ is a normal $2$-subgroup: in an adapted basis the matrices in $K_1$ are block upper triangular with identity diagonal blocks.
The quotient $\bar B=B/K_1$ acts faithfully and completely reducibly on that direct sum, of total dimension $d$.
All hypotheses of \cite[Theorem~1.2]{Maroti2004} hold.
Therefore
\[
 k_{2'}(\Out(P))=k_{2'}(B)=k_{2'}(\bar B)
 \leq k(\bar B)\leq2^{5d}\leq2^{10d}.\qedhere
\]
\end{proof}

\begin{lemma}\label{lem:outer-actions}
 Let $P$ be a finite $2$-group, and let $H$ be a fixed finite group of order $2^cq$, where $q>1$ is odd.
 Then
 \[
\#\bigl(\Hom(H,\Out(P))/\Out(P)\bigr)
 \leq2^{10d(P)}|\Aut(P)|^{c+\Omega(q)-1}.
\]
\end{lemma}
\begin{proof}
Choose $x\in H$ of odd prime order $\ell$.
Starting with $H_1=\langle x\rangle$, adjoin one generator at a time along a strict chain $H_1<\cdots<H_t=H$.
Since
\[
 |H|/\ell=\prod_{i=1}^{t-1}[H_{i+1}:H_i],
\]
each index has at least one prime factor, and $t-1\leq c+\Omega(q)-1$.
Under a homomorphism, $x$ has an odd-order image.
First conjugate that image into one of the at most $2^{10d(P)}$ representatives supplied by Lemma~\ref{lem:odd-outer-classes}, then choose the other generator images in $\Out(P)$.
Relations and residual centralizer conjugacy can only reduce this upper bound, and $|\Out(P)|\leq|\Aut(P)|$.
\end{proof}

\subsection{Extensions with prescribed outer action}\label{subsec:outer-extensions}

\begin{lemma}\label{lem:outer-extensions}
 For fixed finite groups $N,H$ and a prescribed outer action $\theta:H\to\Out(N)$, the number of isomorphism classes of extensions $1\to N\to G\to H\to1$ inducing $\theta$ is at most
 \[
|Z(N)|^{|H|^2}.
\]
 For $|N|=2^b$ and fixed $H$, this is $2^{O(b)}$.
\end{lemma}
\begin{proof}
Choose representatives $\sigma_h\in\Aut(N)$ of $\theta(h)$, with $\sigma_1=1$, and use $\Inn(u)(n)=unu^{-1}$.
In any realizing extension choose a section $s:H\to G$.
If conjugation by $s(h)$ is $\Inn(n_h)\sigma_h$, replace $s(h)$ by $n_h^{-1}s(h)$; it then acts by exactly $\sigma_h$.
Make these choices independently, with $s(1)=1$.
The resulting factor set satisfies
\[
 \Inn(u(h,k))=\sigma_h\sigma_k\sigma_{hk}^{-1}.
\]
For each pair $(h,k)$, the possible values of $u(h,k)$ lie in a single coset of $Z(N)$, of size at most $|Z(N)|$.
The multiplication is
\[
(n,h)(n',k)=(n\sigma_h(n')u(h,k),hk).
\]
There are $|H|^2$ pairs.
Counting all choices of the values $u(h,k)$ gives an upper bound, since some choices may fail associativity or normalization.
Two extensions with identical data are isomorphic through their normal forms $ns(h)$.
This proves the bound.
\end{proof}

\subsection{Fixed odd cofactors}\label{subsec:odd-cofactors}

\begin{proposition}\label{prop:fixed-odd-cofactor}
 For every fixed odd integer $q>1$,
 \begin{equation}\label{eq:fixed-cofactor-bound}
 \log_2f(2^aq)\leq\lambda(a)
       +\frac{2(\Omega(q)-1)}9a^2+O(a).
 \end{equation}
\end{proposition}
\begin{proof}
For $|G|=2^aq$, choose a Sylow $2$-subgroup $P$ and let $N=\operatorname{core}_G(P)$, $[P:N]=2^c$, and $H=G/N$.
The coset action embeds $H$ in the symmetric group on $q$ letters, so
\[
0\leq c\leq v_2(q!),\quad |N|=2^{a-c},\quad |H|=2^cq.
\]
In fact $N=O_2(G)$: the core is a normal $2$-subgroup, and every normal $2$-subgroup lies in every Sylow $2$-subgroup.
For fixed $q$ only finitely many $c$ and abstract types of $H$ occur.
We may sum over all groups of order $2^cq$, without counting their embeddings in the symmetric group.

Fix $c,H$ and let $b=a-c$, $k=c+\Omega(q)-1$.
Lemma~\ref{lem:outer-actions} bounds outer actions up to conjugacy by $2^{10d(N)}|\Aut(N)|^k$.
Conjugacy in $\Out(N)$ is legitimate: each conjugating element lifts to $\Aut(N)$, and relabeling the kernel preserves the abstract isomorphism class of the middle group.
For each action, Lemma~\ref{lem:outer-extensions} gives at most $|Z(N)|^{|H|^2}$ extensions.
Since $d(N)\leq a$ and $|Z(N)|\leq2^a$, the contribution of all such kernels is at most
\[
2^{O(a)}\sum_{|N|=2^{a-c}}|\Aut(N)|^k
 =2^{O(a)}M_k(a-c).
\]
More explicitly,
\begin{equation}\label{eq:cofactor-sum}
 f(2^aq)\leq\sum_{c=0}^{\min\{a,v_2(q!)\}}
 f(2^cq)\,2^{(10+(2^cq)^2)a}
 M_{c+\Omega(q)-1}(a-c).
\end{equation}
Proposition~\ref{prop:automorphism-moments} gives
\[
 \log_2M_k(a-c)\leq\lambda(a-c)+\frac{2k}9(a-c)^2+O(a).
\]
For fixed $c$, $\lambda(a-c)=\lambda(a)-(2c/9)a^2+O(a)$.
With $k=c+\Omega(q)-1$, the quadratic terms in $c$ cancel, leaving \eqref{eq:fixed-cofactor-bound}.
All moment orders, quotient counts, and linear coefficients in \eqref{eq:cofactor-sum} range over a fixed finite set; taking the largest threshold and summing proves the assertion.
\end{proof}

\subsection{Summation over a dyadic interval}\label{subsec:dyadic-sum}

We use two published estimates with constants uniform over all orders.
If $\mu(n)=\max_{p^a\parallel n}a$, then, after fixing all abstract Sylow types, \cite[Theorem~16.20]{BNV2007} bounds the number of groups of order $n$ by
\begin{equation}\label{eq:fixed-sylow-bound}
 n^{\kappa\mu(n)+C_B},\quad \kappa=97/4,\quad C_B=278852.
\end{equation}
Also \cite[Corollary~16.21]{BNV2007} supplies an absolute constant $C$ for which, for every $n\geq2$,
\begin{equation}\label{eq:uniform-group-bound}
 \log_2f(n)\leq(\log_2n)
 \left(\frac2{27}\mu(n)^2+C\mu(n)^{\frac{3}{2}}\right).
\end{equation}
This estimate is uniform over the prime divisors of $n$; we do not apply the fixed-prime estimate of Corollary~\ref{cor:logarithmic-count} uniformly over those primes.
The alternative bound $n^{75\mu(n)+16}$ from \cite[p.~204]{Pyber1993} is also sufficient, with the cutoff specified below.

\begin{proof}[\rm\textbf{Proof of Theorem~\ref{thm:main}}]
Every non-power of two $n<2^{m+1}$ has the form
\[
 n=2^{m-r}q,\quad1\leq r\leq m,\quad
 q>1\text{ odd},\quad q<2^{r+1}.
\]
Summing \eqref{eq:fixed-sylow-bound} over Sylow types gives
\begin{equation}\label{eq:sylow-factor-count}
 f(n)\leq n^{\kappa\mu(n)+C_B}P_2(m-r)f(q).
\end{equation}
The product of the numbers of odd Sylow types is the number of nilpotent groups of order $q$, and is at most $f(q)$.
By Corollary~\ref{cor:logarithmic-count} and a finite adjustment at small orders, there are absolute $C_2,C_0\geq0$ such that
\[
 \log_2P_2(a)\leq\lambda(a)+C_2a+C_0\quad(a\geq0).
\]
Set $t=\log_2q$.
Since $q$ is odd, $\mu(q)\leq t/\log_2 3<2t/3$, so \eqref{eq:uniform-group-bound} gives uniformly
\[
\log_2f(q)\leq\frac8{243}t^3+O(t^{\frac{5}{2}})
 \leq\frac8{243}r^3+O(r^2+r^{\frac{5}{2}}),
\]
using $t<r+1$.
With $u=r/m\in[0,1]$, the exact identity
\[
\begin{aligned}
 \lambda(m)-\lambda(m-r)-\frac8{243}r^3
 &=\frac{rm^2}{243}\bigl(10+(1-u)(44-10u)\bigr)
       -\frac89rm+\frac49r^2\\
 &\geq\frac{10}{243}rm^2-\frac89rm
 \end{aligned}
\]
follows by expansion and $(1-u)(44-10u)\geq0$.
Also $\mu(n)\leq\log_2n<m+1$, whence
\[
 (\kappa\mu(n)+C_B)\log_2n\leq\kappa m^2+O(m).
\]
Substitution in \eqref{eq:sylow-factor-count} gives
\[
 \log_2f(n)-\lambda(m)
 \leq\kappa m^2-\frac{10}{243}rm^2
      +O(rm+r^2+r^{\frac{5}{2}})+O(m).
\]
Since $r^2\leq rm$ and $r^{5/2}\leq rm^{3/2}$, the first error divided by $rm^2$ tends to zero uniformly for $1\leq r\leq m$.
As $1/25<10/243$, it follows that, uniformly for large $m$,
\[
 \log_2f(n)-\lambda(m)\leq(\kappa-r/25)m^2+O(m).
\]
Fix $R=1000$.
If $R\leq r\leq m$, then $\kappa-r/25\leq-63/4$.
There are fewer than $2^{m+1}$ orders to sum over, so the sum still has a fixed quadratic saving.
Using the alternative bound of \cite{Pyber1993}, take $R=2000$ instead; then $75-r/25\leq-5$.

If $1\leq r<R$, the possible pairs $(r,q)$ with $q<2^{r+1}$ form a fixed finite set.
Only for these cofactors do we invoke Proposition~\ref{prop:fixed-odd-cofactor}, obtaining
\[
 \log_2f(2^{m-r}q)
 \leq\lambda(m)-\frac29(r-\Omega(q)+1)m^2+O(m).
\]
The inequalities $3^{\Omega(q)}\leq q<2^{r+1}$ imply $\Omega(q)\leq r$, so the saving is at least $2/9$.
Finiteness allows one common threshold and error constant.
Combining the two ranges gives, for some $c_1>0$,
\begin{equation}\label{eq:nonbinary-order-sum}
 \sum_{\substack{n<2^{m+1}\\n\text{ not a power of }2}}f(n)
 \leq2^{\lambda(m)-c_1m^2}.
\end{equation}

For smaller powers of two, including $2^0=1$, eventual monotonicity of $\lambda$ and the all-range bound for $P_2$ give
\[
 \sum_{0\leq i<m}P_2(i)
 \leq O(1)+m\,2^{\lambda(m-1)+C_2m+C_0}.
\]
Since $\lambda(m)-\lambda(m-1)=2m^2/9+O(m)$, there is $c_2>0$ such that
\begin{equation}\label{eq:smaller-binary-order-sum}
 \sum_{0\leq i<m}P_2(i)\leq2^{\lambda(m)-c_2m^2}.
\end{equation}
For every $x\in[2^m,2^{m+1})$, the count $F(x)$ includes all groups of order $2^m$; the contribution from groups of all other orders is controlled by \eqref{eq:nonbinary-order-sum}--\eqref{eq:smaller-binary-order-sum}.
Since $P_2(m)\geq2^{\lambda(m)}$,
\[
 0\leq\frac{F(x)}{P_2(m)}-1
 \leq2^{-c_1m^2}+2^{-c_2m^2}.
\]
Decreasing the exponent constant proves \eqref{eq:dyadic-asymptotic}.
Finally $m\leq\log_2x<m+1$ and $m^2\geq(\log_2x)^2/4$ for $m\geq1$, proving \eqref{eq:cumulative-asymptotic}.

It remains to prove the assertion about nilpotency class.
Let $b(m)$ be the number of abelian groups of order $2^m$.
The classification of finite abelian groups gives $b(m)\leq2^{m-1}$ for $m\geq1$.
Since $P_2(m)\geq2^{\lambda(m)}$, we have
\[
\frac{b(m)}{P_2(m)}\leq2^{m-1-\lambda(m)}.
\]
The number of groups of order $2^m$ and class exactly two is $Q_2(m)-b(m)$.
Hence
\[
\begin{aligned}
0\leq1-\frac{Q_2(m)-b(m)}{F(x)}
&\leq\frac{F(x)-P_2(m)}{P_2(m)}
   +\frac{P_2(m)-Q_2(m)}{P_2(m)}
   +\frac{b(m)}{P_2(m)}\\
&\leq O(2^{-cm^2})+2^{-\frac{m^2}{300}}+2^{m-1-\lambda(m)},
\end{aligned}
\]
where $(P_2(m)-Q_2(m))/P_2(m)\leq2^{-m^2/300}$ follows from Theorem~\ref{thm:structure}.
Decreasing $c$ if necessary completes the proof.
\end{proof}

\begin{corollary}\label{cor:cumulative-domination}
There are absolute constants $K,c>0$ and an integer $M\geq1$ such that, for every integer $m\geq M$,
\begin{equation}\label{eq:cumulative-domination}
\sum_{1\leq n<2^m}f(n)\leq K2^{-cm^2}f(2^m)<f(2^m).
\end{equation}
Consequently, $f(n)<f(2^m)$ for $1\leq n<2^m$, and $f$ has its unique maximum on $\{1,\ldots,2^m\}$ at $2^m$.
\end{corollary}

\begin{proof}[\rm\textbf{Proof of Corollary~\ref{cor:cumulative-domination}}]
Theorem~\ref{thm:main} at $x=2^m$ gives, beyond a fixed threshold,
\[
 0\leq\sum_{n<2^m}f(n)=F(2^m)-f(2^m)
 \leq K2^{-cm^2}f(2^m).
\]
Increase the threshold to $M$ so that $K2^{-c M^2}<1$.
This proves \eqref{eq:cumulative-domination}.
Since each $f(n)$ is nonnegative, we have $f(n)<f(2^m)$ for $1\leq n<2^m$, as required.
\end{proof}

\end{document}